\documentclass[12pt,british,a4paper,reqno]{amsart}

	\usepackage[T1]{fontenc}
	\usepackage[latin9]{inputenc}
	\usepackage[british]{babel}

	\usepackage{amsmath}
	\usepackage{amssymb}
	\usepackage{amsthm}
	\usepackage{bbm}
	\usepackage{braket}
	\usepackage{xcolor}
	\usepackage[shortlabels]{enumitem}
	\usepackage{fancyhdr}
	\usepackage{graphicx}
	\usepackage{ifsym}
	\usepackage{indentfirst}
	\usepackage{lmodern}
	\usepackage{mathrsfs}
	\usepackage{mathtools}
	\usepackage{oplotsymbl} 
	\usepackage{parskip}
	\usepackage{proof}
	\usepackage{qtree}
	\usepackage{scalerel}
	\usepackage{setspace}
	\usepackage{stmaryrd}
	\usepackage{tensor}
	\usepackage{tikz}
	\usepackage{tikz-cd}
	\usepackage{wasysym}
	\usepackage{xparse}

	\usepackage[hyperfootnotes=false,hidelinks]{hyperref}
		\usepackage[capitalise]{cleveref}
		\usepackage{bookmark}
	\usepackage{geometry}

	\bookmarksetup{numbered,open}

	\renewcommand{\geq}{\geqslant}
	\renewcommand{\leq}{\leqslant}
	\renewcommand{\phi}{\varphi}
	\renewcommand{\Im}{\operatorname{Im}\nolimits}
	\renewcommand{\sp}{\mathrm{sp}}

	\providecommand{\corollaryname}{Corollary}
	\providecommand{\definitionname}{Definition}
	\providecommand{\examplename}{Example}
	\providecommand{\lemmaname}{Lemma}
	\providecommand{\notationname}{Notation}
	\providecommand{\propositionname}{Proposition}
	\providecommand{\remarkname}{Remark}
	\providecommand{\theoremname}{Theorem}
	\providecommand{\setupname}{Setup}
	\providecommand{\conjecturename}{Conjecture}
	\providecommand{\questionname}{Question}
	\providecommand{\claimname}{Claim}

	\theoremstyle{plain}
		\newtheorem{thm}{\protect\theoremname}[section] 
		\newtheorem{thmx}{Theorem}
		\newtheorem{prop}[thm]{\protect\propositionname}
		\newtheorem{lem}[thm]{\protect\lemmaname}

	\theoremstyle{definition}
		\newtheorem{defn}[thm]{\protect\definitionname}

		\newtheorem{example}[thm]{\protect\examplename}
		\newtheorem{setup}[thm]{\setupname}

	\theoremstyle{remark}
		\newtheorem{rem}[thm]{\protect\remarkname}
		
	\numberwithin{figure}{section}
	\numberwithin{equation}{section}

	\usetikzlibrary{matrix,arrows,decorations.pathmorphing,positioning,decorations.pathreplacing}
	\tikzset{commutative diagrams/.cd, 
		mysymbol/.style = {start anchor=center, end anchor = center, draw = none}}
	\tikzcdset{every label/.append style = {font = \footnotesize}}

	\newcommand{\BA}{\mathbb{A}}
	
	\newcommand{\BD}{\mathbb{D}}
	\newcommand{\BE}{\mathbb{E}}

	\newcommand{\BS}{\mathbb{S}}

	\newcommand{\CA}{\mathcal{A}}
	\newcommand{\CB}{\mathcal{B}}
	\newcommand{\CC}{\mathcal{C}}
	\newcommand{\DD}{\mathcal{D}}

	\newcommand{\CI}{\mathcal{I}}

	\newcommand{\CT}{\mathcal{T}}

	\newcommand{\CX}{\mathcal{X}}
	\newcommand{\CY}{\mathcal{Y}}
	\newcommand{\CZ}{\mathcal{Z}}

	\newcommand{\SR}{\mathscr{R}}

		\newcommand{\Ab}{\operatorname{\mathsf{Ab}}\nolimits}
		\newcommand{\add}{\operatorname{\mathsf{add}}\nolimits}

		\newcommand{\ind}{\operatorname{\mathsf{indec}}\nolimits}

		\newcommand{\skel}{\mathrm{skel}}
		\newcommand{\op}{\mathrm{op}}
		
		\newcommand{\field}{k}
		
		\newcommand{\sse}{\subseteq}
		\newcommand{\spse}{\supseteq}

		\newcommand{\obj}{\operatorname{obj}\nolimits}

		\newcommand{\Ker}{\operatorname{Ker}\nolimits}
		\newcommand{\Cok}{\operatorname{Coker}\nolimits}

		\newcommand{\iso}{\cong}

		\newcommand{\Hom}{\operatorname{Hom}\nolimits}

		\newcommand{\onto}{\rightarrow\mathrel{\mkern-14mu}\rightarrow}

		\newcommand{\idfunc}[1]{\mathbbm{1}_{#1}}
		\newcommand{\iden}[1]{\tensor[]{\mathrm{id}}{_{#1}}}

		\newcommand{\indxx}[1]{\tensor[]{\operatorname{\mathrm{index}}}{_{#1}}}

		\newcommand{\newindx}[1]{\tensor[]{\operatorname{\mathrm{ind}}}{_{#1}}}

		\newcommand{\Ext}{\operatorname{Ext}\nolimits}
		
		\newcommand{\fs}{\mathfrak{s}}
		
		\newcommand{\sus}{\Sigma} 

		\newcommand{\rmod}[1]{\operatorname{\mathsf{mod}}\nolimits{#1}}
		
		\newcommand{\rMod}[1]{\operatorname{\mathsf{Mod}}\nolimits{#1}}

		\newcommand{\combul}{\raisebox{0.5pt}{\scalebox{0.6}{\(\bullet\)}}}

	\newcommand{\deff}{\coloneqq}
	\newcommand{\eps}{\varepsilon}
	\newcommand{\lan}{\langle}
	\newcommand{\ran}{\rangle}

	\newcommand{\wt}[1]{\widetilde{#1}}
	
	\newcommand{\ol}[1]{\overline{#1}}
	\newcommand{\ul}[1]{\underline{#1}}
	
	\newcommand\restr[2]{{\left.\kern-\nulldelimiterspace#1
						\right|_{#2}}}

\makeatletter

	\renewcommand{\andify}{%
		\nxandlist{\unskip, }{\unskip{} \@@and~}{\unskip \penalty-2 \space \@@and~}}
    
	\renewcommand\author@andify{%
  		\nxandlist {\unskip ,\penalty-1 \space\ignorespaces}%
		{\unskip {} \@@and~}%
		{\unskip \penalty-2 \space \@@and~}
	}

	    \newenvironment{acknowledgements}{%
	    \renewcommand\abstractname{Acknowledgements}
\global\setbox\abstractbox=\vtop \bgroup
\normalfont\Small
\list{}{\labelwidth\z@
\leftmargin3pc \rightmargin\leftmargin
\listparindent\normalparindent \itemindent\z@
\parsep\z@ \@plus\p@

}%
\item[\hskip\labelsep\scshape\abstractname.]%
}{%
\endlist\egroup
\ifx\@setabstract\relax \@setabstracta \fi
}

\def\@setaddresses{\par
    \nobreak \begingroup
    \setstretch{0.5} 
    \footnotesize
    \def\author##1{\nobreak\addvspace\bigskipamount}%
    \def\\{\unskip, \ignorespaces}%
    \interlinepenalty\@M
    \def\address##1##2{\begingroup
        \par\addvspace\bigskipamount\indent
    \@ifnotempty{##1}{(\ignorespaces##1\unskip) }%
    {\scshape\ignorespaces##2}\par\endgroup}%
    \def\curraddr##1##2{\begingroup
    \@ifnotempty{##2}{\nobreak\indent\curraddrname
    \@ifnotempty{##1}{, \ignorespaces##1\unskip}\/:\space
    ##2\par}\endgroup}%
    \def\email##1##2{\begingroup
    \@ifnotempty{##2}{\nobreak\indent\emailaddrname
    \@ifnotempty{##1}{, \ignorespaces##1\unskip}\/:\space
    \ttfamily##2\par}\endgroup}%
    \def\urladdr##1##2{\begingroup
    \def~{\char'\~}%
    \@ifnotempty{##2}{\nobreak\indent\urladdrname
    \@ifnotempty{##1}{, \ignorespaces##1\unskip}\/:\space
    \ttfamily##2\par}\endgroup}%
    \addresses
    \endgroup
}

\let\oldtocsection=\tocsection
\let\oldtocsubsection=\tocsubsection

\renewcommand{\tocsection}[2]{\vspace*{0pt}\hspace{0em}\oldtocsection{#1}{#2}}

    \renewcommand{\tocsubsection}[2]{\vspace*{0pt}\hspace{21pt}\oldtocsubsection{#1}{#2}}

\let\amph\&
\makeatother

\begin{document}

\title{The index with respect to a rigid subcategory of a triangulated category}
\author[J{\o{}}rgensen]{Peter J{\o{}}rgensen}
	\address{
		Department of Mathematics\\
		Aarhus University\\
		8000 Aarhus C\\
		Denmark
	}
    \email{peter.jorgensen@math.au.dk}

\author[Shah]{Amit Shah}
	\address{
		Department of Mathematics\\
		Aarhus University\\
		8000 Aarhus C\\
		Denmark
	}
    \email{amit.shah@math.au.dk}

\date{\today}



{\setstretch{1}\begin{abstract}

Palu defined the index with respect to a cluster tilting object in a suitable triangulated category, in order to better understand the Caldero-Chapoton map that exhibits the connection between cluster algebras and representation theory. We push this further by proposing an index with respect to a contravariantly finite, rigid subcategory, and we show this index behaves similarly to the classical index. 

Let \(\CC\) be a skeletally small triangulated category with split idempotents, which is thus an extriangulated category \((\CC,\BE,\fs)\). 
Suppose \(\CX\) is a contravariantly finite, rigid subcategory in \(\CC\). 
We define the index \(\newindx{\CX}(C)\) of an object \(C\in\CC\) with respect to \(\CX\) as the \(\tensor[]{K}{_{0}}\)-class \([C\tensor[]{]}{_{\CX}}\) in Grothendieck group \(\tensor[]{K}{_{0}}(\CC,\tensor[]{\BE}{_{\CX}},\tensor[]{\fs}{_{\CX}})\) of the relative extriangulated category \((\CC,\tensor[]{\BE}{_{\CX}},\tensor[]{\fs}{_{\CX}})\). 
By analogy to the classical case, we give an additivity formula with error term for \(\newindx{\CX}\) on triangles in \(\CC\).

In case \(\CX\) is contained in another suitable subcategory \(\CT\) of \(\CC\), there is a surjection 
\(Q\colon \tensor[]{K}{_{0}}(\CC,\tensor[]{\BE}{_{\CT}},\tensor[]{\fs}{_{\CT}}) \onto \tensor[]{K}{_{0}}(\CC,\tensor[]{\BE}{_{\CX}},\tensor[]{\fs}{_{\CX}})\). Thus, in order to describe \(\tensor[]{K}{_{0}}(\CC,\tensor[]{\BE}{_{\CX}},\tensor[]{\fs}{_{\CX}})\), it suffices to determine \(\tensor[]{K}{_{0}}(\CC,\tensor[]{\BE}{_{\CT}},\tensor[]{\fs}{_{\CT}})\) and \(\Ker Q\). We do this under certain assumptions. 
\end{abstract}}

\maketitle


\section{Introduction}
\label{sec:introduction}

Cluster algebras were introduced by Fomin--Zelevinsky in \cite{FominZelevinsky-cluster-algebras-I} and have had an impact not only in mathematical fields, such as combinatorics and geometry, but also in seemingly more distant disciplines, especially particle physics where there is a strong connection to scattering amplitudes. The introduction of the cluster category by Buan--Marsh--Reineke--Reiten--Todorov \cite{BMRRT-cluster-combinatorics} has established a substantial bond between cluster theory and representation theory. Cluster categories give ``categorifications'' of cluster algebras via a certain map introduced by Caldero--Chapoton \cite{CalderoChapoton-cluster-algebras-as-hall-algebras-of-quiver-representations}. In suitable contexts, Palu showed in \cite{Palu-cluster-characters-for-2-calabi-yau-triangulated-categories} that a version proposed by Caldero--Keller in \cite{CalderoKeller-from-triangulated-categories-to-cluster-algebras-I} of the original Caldero-Chapoton map induces a bijection between the indecomposable rigid objects in the cluster category and the cluster variables in the corresponding cluster algebra 
(see also \cite{BuanMarshReitenTodorov-Clusters-and-seeds-in-acyclic-cluster-algebras}, \cite{CalderoKeller-from-triangulated-categories-to-cluster-algebras-II}).
Results of this kind have generated vast interest in the Caldero-Chapoton map in various settings; see e.g.\
\cite{CalderoChapoton-cluster-algebras-as-hall-algebras-of-quiver-representations}, 
\cite{CalderoKeller-from-triangulated-categories-to-cluster-algebras-II}, 
\cite{Demonet-categorification-of-skew-symmetrizable-cluster-algebras}, 
\cite{DerksenWeymanZelevinsky-quivers-with-potentials-and-their-representations-I}, 
\cite{DerksenWeymanZelevinsky-quivers-with-potentials-and-their-representations-II}, 
\cite{FuKeller-on-cluster-algebras-with-coefficients-and-2-calabi-Yau-categories}, 
\cite{HolmJorgensen-generalized-friezes-and-a-modified-caldero-chapoton-map-depending-on-a-rigid-object-1}, 
\cite{HolmJorgensen-generalized-friezes-and-a-modified-caldero-chapoton-map-depending-on-a-rigid-object-2}, 
\cite{JorgensenPalu-a-caldero-chapoton-map-for-infinite-clusters}, 
\cite{Plamondon-cluster-characters-for-cluster-categories-with-infinite-dimensional-morphism-spaces}, 
\cite{ZhouZhu-cluster-algebras-arising-from-cluster-tubes}.

A particularly fruitful aspect of this relationship between cluster theory and representation theory has been the study of cluster tilting objects. In order to better understand the Caldero-Chapoton map, Palu defined the index with respect to a cluster tilting object in a suitable triangulated category. Although initially used as a tool, this index is important and interesting in its own right. 
For example: 
\begin{itemize}[leftmargin=20pt, labelwidth=10pt, labelsep=5pt, noitemsep]

	\item Dehy--Keller \cite{DehyKeller-on-the-combinatorics-of-rigid-objects-in-2-calabi-yau-categories} 
	showed that the index induces a bijection between the set of isomorphism classes of rigid indecomposable objects reachable from a cluster tilting subcategory and the corresponding \textbf{g}-vectors (see \cite[proof of Thm.\ 2.10]{JorgensenYakimov-c-vectors-of-2-calabi-yau-categories-and-borel-subalgebras}); 
	\item Grabowski--Pressland  \cite{Grabowski-graded-cluster-algebras}, \cite{GrabowskiPressland-graded-frobenius-cluster-categories} used the index to define graded cluster categories; and 
	\item J{\o}rgensen \cite{Jorgensen-tropical-friezes-and-the-index-in-higher-homological-algebra} has generalised the index to a higher homological algebra setting, using it to construct higher-dimensional tropical friezes (see also \cite{Guo-on-tropical-friezes-associated-with-dynkin-diagrams}).

\end{itemize}

Our initial setup is very general. Suppose \(\CC\) is a skeletally small, idempotent complete, triangulated category with suspension functor \(\sus\). Suppose \(\CT\sse\CC\) is a \emph{(2-)cluster tilting} subcategory (see Definition~\ref{def:n-cluster-tilting}) and let \(C\in\CC\) be an object. 
Since \(\CT\) is cluster tilting, there is a triangle of the form 
\(\begin{tikzcd}[column sep=0.5cm]
\tensor*[]{T}{_{1}^{C}} \arrow{r}& \tensor*[]{T}{_{0}^{C}} \arrow{r}& C \arrow{r} &\sus \tensor*[]{T}{_{1}^{C}},
\end{tikzcd}\)
with \(\tensor*[]{T}{_{i}^{C}}\in\CT\). 
The \emph{index of \(C\) with respect to \(\CT\)}, in the sense of Palu \cite[Sec.\ 2.1]{Palu-cluster-characters-for-2-calabi-yau-triangulated-categories}, is the element 
\([\tensor*[]{T}{_{0}^{C}}\tensor*[]{]}{_{\CT}^{\sp}}-[\tensor*[]{T}{_{1}^{C}}\tensor*[]{]}{_{\CT}^{\sp}}\) of the split Grothendieck group \(\tensor*[]{K}{_{0}^{\sp}}(\CT)\) of \(\CT\).
Note that the way in which these \(\tensor*[]{T}{_{i}^{C}}\) are found really relies on \(\CT\) being cluster tilting, so there seems to be no immediate way to extend this theory to the more general rigid subcategory case. 

Padrol--Palu--Pilaud--Plamondon recently showed in \cite{PadrolPaluPilaudPlamondon-associahedra-for-finite-type-cluster-algebras-and-minimal-relations-between-g-vectors} that \(\tensor*[]{K}{_{0}^{\sp}}(\CT)\) shows up in another way using the theory of \emph{extriangulated categories} of Nakaoka--Palu \cite{NakaokaPalu-extriangulated-categories-hovey-twin-cotorsion-pairs-and-model-structures}. 
Since \(\CC\) is a triangulated category, we can view it as an extriangulated category \((\CC,\BE,\fs)\) (see Example~\ref{example:triangulated-category-is-extriangulated}). In this notation, \(\BE\colon \tensor[]{\CC}{^{\op}}\times\CC \to \Ab\) is a biadditive functor to the category \(\Ab\) of abelian groups given by 
\(\BE(C,A) = \CC(C,\sus A)\) on objects. Each element \(\gamma \in \BE(C,A)\) is part of a so-called \emph{extriangle} in \((\CC,\BE,\fs)\) of the form  
\(\begin{tikzcd}[column sep=0.5cm]
A \arrow{r}{} & B \arrow{r}{} & C \arrow[dashed]{r}{\gamma} & {},
\end{tikzcd}\)
where 
\(\begin{tikzcd}[column sep=0.5cm]
A \arrow{r}{} & B \arrow{r}{} & C \arrow{r}{\gamma} & \sus A
\end{tikzcd}\)
is a distinguished triangle in \(\CC\). 

In order to recover \(\tensor*[]{K}{_{0}^{\sp}}(\CT)\) as the Grothendieck group of a certain extriangulated category, we need to tweak this structure just described on \(\CC\) using \(\CT\). 
The subcategory \(\CT\) induces a \emph{relative} extriangulated structure on \(\CC\) in the sense of \cite[Sec.\ 3.2]{HerschendLiuNakaoka-n-exangulated-categories-I-definitions-and-fundamental-properties} (see Subsection~\ref{subsec:extriangulated-categories}), which yields an extriangulated category \((\CC,\tensor[]{\BE}{_{\CT}},\tensor[]{\fs}{_{\CT}})\), where
\begin{equation}
\label{eqn:relative-extriangulated-structure-wrt-T-in-intro}
\tensor[]{\BE}{_{\CT}}(C,A)
	= \Set{ \gamma\in\BE(C,A) | 
	 	  \gamma \circ \xi = 0 \text{ for all } \xi\colon T \to C \text{ with } T\in\CT
		  }. 
\end{equation}
That is, the extriangles in \((\CC,\tensor[]{\BE}{_{\CT}},\tensor[]{\fs}{_{\CT}})\) are of the form 
\(\begin{tikzcd}[column sep=0.5cm]
A \arrow{r}{} & B \arrow{r}{} & C \arrow[dashed]{r}{\gamma} & {},
\end{tikzcd}\)
where again 
\(\begin{tikzcd}[column sep=0.5cm]
\hspace{-5pt}A \arrow{r}{} & B \arrow{r}{} & C \arrow{r}{\gamma} & \sus A
\end{tikzcd}\)
is a triangle in \(\CC\) but now where we demand that \(\restr{\CC(-,\gamma)}{\CT} = 0\). 
The \emph{Grothendieck group of \((\CC,\tensor[]{\BE}{_{\CT}},\tensor[]{\fs}{_{\CT}})\)} is defined to be 
\begin{equation}
\label{eqn:grothendieck-group-of-cluster-tilting-subcategory-in-intro}
\tensor[]{K}{_{0}}(\CC,\tensor[]{\BE}{_{\CT}},\tensor[]{\fs}{_{\CT}}) = \tensor*[]{K}{_{0}^{\sp}}(\CC)   /   \tensor[]{\CI}{_{\CT}}
\end{equation}
where 
\begin{equation}
\label{eqn:relations-for-grothendieck-group-of-cluster-tilting-subcategory-in-intro}
\tensor[]{\CI}{_{\CT}}  \deff  
	\Braket { [A\tensor*[]{]}{_{\CC}^{\sp}}-[B\tensor*[]{]}{_{\CC}^{\sp}}+[C\tensor*[]{]}{_{\CC}^{\sp}} | 
		 \begin{aligned}
			&\hspace{-5pt}\begin{tikzcd}[column sep=0.5cm,ampersand replacement=\&]
			A \arrow{r}{} \& B \arrow{r}{} \& C \arrow[dashed]{r}{\gamma} \& {}
			\end{tikzcd}\text{is an}\\[-7pt]
			&\text{extriangle in } (\CC,\tensor[]{\BE}{_{\CT}},\tensor[]{\fs}{_{\CT}})
		\end{aligned}
		};
\end{equation}
see \cite{Haugland-the-grothendieck-group-of-an-n-exangulated-category}, 
\cite{ZhuZhuang-grothendieck-groups-in-extriangulated-categories}, 
or Definition~\ref{def:Grothendieck-group-of-extriangulated-category} and Remark~\ref{rem:grothendieck-group-of-relative-structure}. 
For \(C\in\CC\), we denote its class in \(\tensor[]{K}{_{0}}(\CC,\tensor[]{\BE}{_{\CT}},\tensor[]{\fs}{_{\CT}})\) by \([C\tensor[]{]}{_{\CT}}\).
It then follows from \cite[Prop.\ 4.11]{PadrolPaluPilaudPlamondon-associahedra-for-finite-type-cluster-algebras-and-minimal-relations-between-g-vectors} that there is an isomorphism 
\begin{equation}
\label{eqn:isomorphism-of-grothendieck-groups-for-cluster-tilting-T}
\begin{aligned}[t]
\tensor[]{K}{_{0}}(\CC,\tensor[]{\BE}{_{\CT}},\tensor[]{\fs}{_{\CT}}) 		&\overset{\iso}{\longrightarrow}		\tensor*[]{K}{_{0}^{\sp}}(\CT) \\[5pt]
[C\tensor[]{]}{_{\CT}}					&\longmapsto					[\tensor*[]{T}{_{0}^{C}}\tensor*[]{]}{_{\CT}^{\sp}}-[\tensor*[]{T}{_{1}^{C}}\tensor*[]{]}{_{\CT}^{\sp}}  \\[5pt]
[T\tensor[]{]}{_{\CT}}					&\longmapsfrom				[T\tensor*[]{]}{_{\CT}^{\sp}}
\end{aligned}
\end{equation} 
induced by the index of Palu. 
We prove an \(n\)-cluster tilting analogue of this isomorphism 
in Section~\ref{sec:descriptions-of-K0CXEXsX}; see Theorem~\ref{thmx:analogue-of-PPPP-Prop-4-11} below. 

Since the index is a powerful invariant, the isomorphism of Padrol--Palu--Pilaud--Plamon\-don is a striking result that demonstrates the significance of extriangulated categories: 
by transporting along the isomorphism \eqref{eqn:isomorphism-of-grothendieck-groups-for-cluster-tilting-T}, we can think of the equivalence class \([C\tensor[]{]}{_{\CT}}\in \tensor[]{K}{_{0}}(\CC,\tensor[]{\BE}{_{\CT}},\tensor[]{\fs}{_{\CT}})\) also as an index of \(C\) with respect to \(\CT\). Moreover, rather than needing to pass to the split Grothendieck group \(\tensor*[]{K}{_{0}^{\sp}}(\CT)\) of \(\CT\), we can instead adjust the extriangulated structure on \(\CC\) itself. This is something that would be impossible if we were to stay strictly in the realm of triangulated categories.

We push this perspective further, using that any subcategory can induce an extriangulated structure on \(\CC\). In particular, suppose \(\CX\) is a suitable contravariantly finite, rigid subcategory (see Setup~\ref{setup:1}), and consider the extriangulated category \((\CC,\tensor[]{\BE}{_{\CX}},\tensor[]{\fs}{_{\CX}})\) and its Grothendieck group \(\tensor[]{K}{_{0}}(\CC,\tensor[]{\BE}{_{\CX}},\tensor[]{\fs}{_{\CX}})\). 
Note that 
\eqref{eqn:relative-extriangulated-structure-wrt-T-in-intro}, \eqref{eqn:grothendieck-group-of-cluster-tilting-subcategory-in-intro} and \eqref{eqn:relations-for-grothendieck-group-of-cluster-tilting-subcategory-in-intro} above do not use any special properties of \(\CT\), so 
\((\CC,\tensor[]{\BE}{_{\CX}},\tensor[]{\fs}{_{\CX}})\), \(\tensor[]{\CI}{_{\CX}}\) and \(\tensor[]{K}{_{0}}(\CC,\tensor[]{\BE}{_{\CX}},\tensor[]{\fs}{_{\CX}})\) are all defined similarly. 
For an object \(C\in\CC\) consider its equivalence class \([C\tensor[]{]}{_{\CX}}\in \tensor[]{K}{_{0}}(\CC,\tensor[]{\BE}{_{\CX}},\tensor[]{\fs}{_{\CX}})\). 
We show that this assignment behaves 
like 
the classical index with respect to a cluster tilting subcategory. 
Hence, we may consider 
\(\newindx{\CX}(-) \deff [-\tensor[]{]}{_{\CX}}\) 
to be the \emph{index with respect to \(\CX\)}; see Definition~\ref{def:relative-index-wrt-X}.
In particular, we show that \(\newindx{\CX}\) is additive on triangles in \(\CC\) up to an error term; see Theorem~\ref{thmx:X-index-additive-on-triangles-with-error-term} below, or Proposition~\ref{prop:definition-of-theta-sub-CX} and Theorem~\ref{thm:X-index-additive-on-triangles-with-error-term}. 
We denote by \(\rmod{\CX}\) the category of all finitely presented, additive functors \(\tensor[]{\CX}{^{\op}}\to \Ab\); see Subsection~\ref{subsec:conventions-notation}\ref{item:functor-categories}. 

\begin{thmx}
\label{thmx:X-index-additive-on-triangles-with-error-term}
There is a group homomorphism \(\tensor[]{\theta}{_{\CX}} \colon \tensor[]{K}{_{0}}(\rmod{\CX}) \to \tensor[]{K}{_{0}}(\CC,\tensor[]{\BE}{_{\CX}},\tensor[]{\fs}{_{\CX}})\), such that if 
\(\begin{tikzcd}[column sep=0.5cm]
	A \arrow{r}{}& B \arrow{r}{}& C \arrow{r}{\gamma}& \sus A
\end{tikzcd}\) 
is a triangle in \(\CC\), then 
\[
\newindx{\CX}(A) - \newindx{\CX}(B) + \newindx{\CX}(C) 
	=  \tensor[]{\theta}{_{\CX}}\big( [\Im ( \restr{\CC(-,\gamma)}{\CX} )  ]\big).
\]
\end{thmx}

Note that an extriangle
\(\begin{tikzcd}[column sep=0.5cm]
	A \arrow{r}{}& B \arrow{r}{}& C \arrow[dashed]{r}{\gamma}& {}
\end{tikzcd}\)
in \((\CC,\tensor[]{\BE}{_{\CX}},\tensor[]{\fs}{_{\CX}})\) satisfies \(\restr{\CC(-,\gamma)}{\CX} = 0\). 
Thus, it is a special case of 
Theorem~\ref{thmx:X-index-additive-on-triangles-with-error-term} that 
\(\newindx{\CX}\) is \emph{additive on extriangles in \((\CC,\tensor[]{\BE}{_{\CX}},\tensor[]{\fs}{_{\CX}})\)} (see Definition~\ref{def:additive-function-on-extriangulated-category} and Proposition~\ref{lemma7}).

Suppose now that \(\CC\) is a skeletally small, \(\field\)-linear (where \(\field\) is a field), \(\Hom\)-finite, idempotent complete, triangulated category. 
Note that the index with respect to \(\CX\) induces a homomorphism 
\(
\tensor*[]{K}{_{0}^{\sp}}(\CC) \to \tensor[]{K}{_{0}}(\CC,\tensor[]{\BE}{_{\CX}},\tensor[]{\fs}{_{\CX}})
\)
of abelian groups, which we also denote by \(\newindx{\CX}\). 
Another benefit of 
this 
index 
is its compatibility with containments of subcategories. 
Indeed, if \(\CX \sse \CT\), then it follows that \(\tensor[]{\CI}{_{\CT}}\sse\tensor[]{\CI}{_{\CX}}\). There is thus an induced surjective homomorphism 
\(Q \colon \tensor[]{K}{_{0}}(\CC,\tensor[]{\BE}{_{\CT}},\tensor[]{\fs}{_{\CT}}) \onto \tensor[]{K}{_{0}}(\CC,\tensor[]{\BE}{_{\CX}},\tensor[]{\fs}{_{\CX}})\) of abelian groups, such that 
\(\newindx{\CX} = Q \circ \newindx{\CT}\) as abelian group homomorphisms \(\tensor*[]{K}{_{0}^{\sp}}(\CC) \to \tensor[]{K}{_{0}}(\CC,\tensor[]{\BE}{_{\CX}},\tensor[]{\fs}{_{\CX}})\) (see Remark~\ref{remark19}). 
In particular, 
\begin{equation}
\label{eqn:K0X-is-quotient-of-K0T-in-intro}
\tensor[]{K}{_{0}}(\CC,\tensor[]{\BE}{_{\CX}},\tensor[]{\fs}{_{\CX}}) \iso \tensor[]{K}{_{0}}(\CC,\tensor[]{\BE}{_{\CT}},\tensor[]{\fs}{_{\CT}}) / \Ker Q. 
\end{equation}
Thus, to determine \(\tensor[]{K}{_{0}}(\CC,\tensor[]{\BE}{_{\CX}},\tensor[]{\fs}{_{\CX}})\) reduces to giving a description of 
\(\tensor[]{K}{_{0}}(\CC,\tensor[]{\BE}{_{\CT}},\tensor[]{\fs}{_{\CT}})\) and 
\(\Ker Q\). 
In Section~\ref{sec:descriptions-of-K0CXEXsX} we give several presentations of \(\Ker Q\). 
A general description in terms of \(\tensor[]{\theta}{_{\CT}}\) (see Theorem~\ref{thmx:X-index-additive-on-triangles-with-error-term}) is given in Proposition~\ref{prop:KerQ-general-description} assuming only that \(\CX\) and \(\CT\) are additive, contravariantly finite, rigid subcategories closed under isomorphisms and direct summands; and a more explicit presentation using simple modules in \(\rmod{\CT}\) is given in Theorem~\ref{thm:KerQ-T-locally-bounded-has-right-almost-split-maps-description} if \(\CT\) is also locally bounded and admits right almost split morphisms. 

If \(\CT\) is \emph{\(n\)-cluster tilting} (see Definition~\ref{def:n-cluster-tilting}) and locally bounded, and \(\CC\) admits a Serre functor, then Theorem~\ref{thm:Ker-Q-in-terms-of-AR-n-plus-2-angles} describes \(\Ker Q\) in terms of Auslander-Reiten \((n+2)\)-angles in \(\CT\). 
When \(n=2\), and under some more mild assumptions, \(\Ker Q\) is isomorphic to the subgroup \(N \leq \tensor*[]{K}{_{0}^{\sp}}(\CT)\) defined in \cite[Def.\ 2.4]{HolmJorgensen-generalized-friezes-and-a-modified-caldero-chapoton-map-depending-on-a-rigid-object-2} via \eqref{eqn:isomorphism-of-grothendieck-groups-for-cluster-tilting-T} above; see Remark~\ref{remark26}. 
In particular, the quotient \(\tensor*[]{K}{_{0}^{\sp}}(\CT)/N\) was used by Holm--J{\o{}}rgensen to define a \emph{modified Caldero-Chapoton map} in \cite{HolmJorgensen-generalized-friezes-and-a-modified-caldero-chapoton-map-depending-on-a-rigid-object-2}, which generalises the original from \cite{CalderoChapoton-cluster-algebras-as-hall-algebras-of-quiver-representations} in different ways. 
In Remark~\ref{remark26} we also recall how 
\[
\tensor[]{K}{_{0}}(\CC,\tensor[]{\BE}{_{\CX}},\tensor[]{\fs}{_{\CX}}) 
	\iso \tensor[]{K}{_{0}}(\CC,\tensor[]{\BE}{_{\CT}},\tensor[]{\fs}{_{\CT}}) / \Ker Q
	\iso \tensor*[]{K}{_{0}^{\sp}}(\CT)/N
\]
has been used in \cite{JorgensenShah-grothendieck-groups-of-d-exangulated-categories-and-a-modified-CC-map} to define the \emph{\(\CX\)-Caldero-Chapoton map} (see \cite[Def.\ 5.3]{JorgensenShah-grothendieck-groups-of-d-exangulated-categories-and-a-modified-CC-map}), which recovers the construction of Holm--J{\o{}}rgensen. 
These generalised Caldero-Chapoton maps are, under certain conditions, generalised friezes; see 
\cite[Def.\ 3.4]{HolmJorgensen-generalized-friezes-and-a-modified-caldero-chapoton-map-depending-on-a-rigid-object-1} and 
\cite[Thm.\ A]{HolmJorgensen-generalized-friezes-and-a-modified-caldero-chapoton-map-depending-on-a-rigid-object-2}. 
They also recover the combinatorial generalised friezes of \cite{BessenrodtHolmJorgensen-generalized-frieze-pattern-determinants-and-higher-angulations-of-polygons}; see also \cite{CanakciJorgensen-driezes-weak-friezes-and-T-paths}.

Due to the isomorphism \eqref{eqn:isomorphism-of-grothendieck-groups-for-cluster-tilting-T} for a 2-cluster tilting subcategory, it is natural to ask if there is a similar isomorphism for an \(n\)-cluster tilting subcategory \(\CT\), where \(n\geq 2\). 
We answer this positively in Section~\ref{sec:descriptions-of-K0CXEXsX} (see Theorem~\ref{thmx:analogue-of-PPPP-Prop-4-11} below), using the \emph{triangulated index with respect to \(\CT\)} (see \cite[Def.\ 3.3]{Jorgensen-tropical-friezes-and-the-index-in-higher-homological-algebra}, or Definition~\ref{def:triangulated-index-wrt-n-cluster-tilting-subcategory}), which is a generalisation of Palu's index seen above.
This gives a higher analogue of \cite[Prop.\ 4.11]{PadrolPaluPilaudPlamondon-associahedra-for-finite-type-cluster-algebras-and-minimal-relations-between-g-vectors}. 
\begin{thmx}[=Theorem~\ref{thm:analogue-of-PPPP-Prop-4-11}]
\label{thmx:analogue-of-PPPP-Prop-4-11}

If \(\CT\) is \(n\)-cluster tilting, then the triangulated index \(\indxx{\CT}\) of \cite{Jorgensen-tropical-friezes-and-the-index-in-higher-homological-algebra} induces an isomorphism 
\(\Psi \colon \tensor[]{K}{_{0}}(\CC,\tensor[]{\BE}{_{\CT}},\tensor[]{\fs}{_{\CT}}) \overset{\iso}{\longrightarrow}  \tensor*[]{K}{_{0}^{\sp}}(\CT)\)
of abelian groups.
\end{thmx}

In the case when \(\CT\) is locally bounded and \(n\)-cluster tilting, 
combining \eqref{eqn:K0X-is-quotient-of-K0T-in-intro} and Theorem~\ref{thmx:analogue-of-PPPP-Prop-4-11} shows that 
\(\tensor[]{K}{_{0}}(\CC,\tensor[]{\BE}{_{\CX}},\tensor[]{\fs}{_{\CX}})\) 
is isomorphic to a quotient of the split Grothendieck group 
\(\tensor*[]{K}{_{0}^{\sp}}(\CT)\); see Remark~\ref{rem:K0CX-is-quotient-of-split-G-group-of-T}.

This paper is organised as follows. 
In Section~\ref{sec:background} we lay out any conventions we adopt and collect the main background material. 
In Section~\ref{sec:relative-index-wrt-rigid-subcategory} we define the index with respect to a rigid subcategory in a triangulated category and prove Theorem~\ref{thmx:X-index-additive-on-triangles-with-error-term}. 
In Section~\ref{sec:descriptions-of-K0CXEXsX} we 
study 
\(\tensor[]{K}{_{0}}(\CC,\tensor[]{\BE}{_{\CX}}, \tensor[]{\fs}{_{\CX}})\) 
via the canonical surjection 
\(Q \colon \tensor[]{K}{_{0}}(\CC,\tensor[]{\BE}{_{\CT}},\tensor[]{\fs}{_{\CT}}) \onto \tensor[]{K}{_{0}}(\CC,\tensor[]{\BE}{_{\CX}},\tensor[]{\fs}{_{\CX}})\). 
In particular, we establish Theorem~\ref{thmx:analogue-of-PPPP-Prop-4-11} in Subsection~\ref{subsec:relation-to-AR-angles}.


\section{Background}
\label{sec:background}

In Subsection~\ref{subsec:conventions-notation} we detail our general conventions and notation. 
In Subsections~\ref{subsec:approximation} and \ref{subsec:localisation} we provide a summary of the theory of approximation and localisation in triangulated categories that we will need in our main sections. 
In Subsection~\ref{subsec:extriangulated-categories} we recall the necessary details about (relative) extriangulated categories that we will use later.


\subsection{Conventions and notation}
\label{subsec:conventions-notation}

For the following, suppose \(\CA\) is an additive category. 

\begin{enumerate}[label=(\roman*)]
	
	\item\label{item:def-additive-category} For us, an \emph{additive subcategory} \(\CB\) of \(\CA\) is a full subcategory which is closed under isomorphisms, direct sums and direct summands. 
	
	\item If \(\CB\sse \CA\) is a subcategory closed under finite direct sums, then \([\CB]\) denotes the (two-sided) ideal of \(\CA\) consisting of morphisms factoring through \(\CB\). In this case, \(\CA/[\CB]\) denotes the corresponding additive quotient. 

	Fix a skeleton \(\tensor[]{\CA}{^{\skel}}\) of \(\CA\). 
	By \(\ind \CA\) we mean the class of indecomposable objects of \(\tensor[]{\CA}{^{\skel}}\). 
	Moreover, if \(\CB\sse\CA\) is an additive subcategory, we impose that a skeleton \(\tensor[]{\CB}{^{\skel}}\) of \(\CB\) is chosen inside \(\tensor[]{\CA}{^{\skel}}\). In particular, this forces \(\ind\CB\sse\ind\CA\).

	\item\label{item:functor-categories} Suppose now that \(\CA\) is also skeletally small. 
	Let \(\Ab\) denote the category of all abelian groups. 
	A (covariant) functor \(M\colon \tensor[]{\CA}{^{\op}}\to \Ab\) is \emph{finitely presented} if there is an exact sequence 
	\(\begin{tikzcd}[column sep=0.5cm]
	\CA(-,A) \arrow{r}& \CA(-,B) \arrow{r}& M \arrow{r}& 0
	\end{tikzcd}\)
	for some objects \(A,B\in\CA\); see \cite[p.\ 155]{Beligiannis-on-the-freyd-cats-of-additive-cats}. 	
	(Note that these functors have been called `coherent' in e.g.\ \cite{Auslander-coherent-functors} and \cite{IyamaYoshino-mutation-in-tri-cats-rigid-CM-mods}. There is a more general notion of a coherent functor (see \cite[Def.\ B.5]{Fiorot-n-quasi-abelian-categories-vs-n-tilting-torsion-pairs}), but this agrees with what we call finitely presented when, for example, \(\CA\) is an idempotent complete triangulated category. See \cite[App.\ B]{Fiorot-n-quasi-abelian-categories-vs-n-tilting-torsion-pairs} and the references therein for more details.) 

	Denote by \(\rmod{\CA}\) the category of finitely presented, additive functors \(\tensor[]{\CA}{^{\op}}\to \Ab\). This category is additive. 
	If \(\CA\) is also idempotent complete and has weak kernels 
	(e.g.\ \(\CA\) is a contravariantly finite, additive subcategory of an idempotent complete, triangulated category), then \(\rmod{\CA}\) is abelian; see e.g.\ \cite[Def.\ 2.9]{IyamaYoshino-mutation-in-tri-cats-rigid-CM-mods}. See \cite[Sec.\ III.2]{Auslander-representation-dimension-of-artin-algebras-QMC} for more details; see also \cite[Sec.\ 2]{Auslander-coherent-functors} and \cite{Auslander-selected-works-part-1}.

	\item\label{item:H-functor-perp-category} 
	Suppose that \(\CA\) is a skeletally small, idempotent complete, triangulated category and that \(\CB\) is a contravariantly finite, additive subcategory of \(\CA\). 
	Then there is a cohomological, covariant functor 
	\begin{center}
	\(
	\begin{tikzcd}[ampersand replacement=\&, column sep=2cm]
		\begin{aligned}[t]
		\CA \\
		A
		\end{aligned} 
	\arrow[two heads]{r}{\tensor[]{H}{_{\CB}}}  \arrow[maps to, yshift={-0.75cm}]{r}\&
		\begin{aligned}[t]
		& \rmod{\CB} \\
		& \restr{\CA(-, A)}{\CB};
		\end{aligned}
	\end{tikzcd}
	\)	
	\end{center}
	see \cite[p.\ 846]{Beligiannis-rigid-objects-triangulated-subfactors-and-abelian-localizations}. 
	Lastly, we set  
	\[
	\tensor[]{\CB}{^{\perp_{0}}} 
		\deff \Ker \tensor[]{H}{_{\CB}} 
		= \set{A\in\CA | \CA(\CB,A) = 0},
	\] 
	considered as an additive subcategory of \(\CA\). 
	Note that \(\tensor[]{\CB}{^{\perp_{0}}}\) is closed under finite direct sums, so \([\tensor[]{\CB}{^{\perp_{0}}}]\) forms an ideal of \(\CA\). 

\end{enumerate}

We typically assume our categories are skeletally small and idempotent complete in Sections~\ref{sec:relative-index-wrt-rigid-subcategory}--\ref{sec:descriptions-of-K0CXEXsX}. See each setup at the beginning of the relevant (sub)section.


\subsection{Approximation in triangulated categories}
\label{subsec:approximation}

In the general setup for our main sections, 
there is always an ambient triangulated category. 
We utilise a collection of results from approximation theory in triangulated categories, especially in Section~\ref{sec:relative-index-wrt-rigid-subcategory}. 
For the convenience of the reader we recall them here, following Beligiannis \cite{Beligiannis-rigid-objects-triangulated-subfactors-and-abelian-localizations}.

\begin{setup}
\label{setup:approximation}

Let \(\CC\) be a skeletally small, idempotent complete, triangulated category. 
Suppose \(\CX\sse\CC\) is an additive (see Subsection~\ref{subsec:conventions-notation}\ref{item:def-additive-category}), contravariantly finite, rigid subcategory.

\end{setup}

The subcategory \(\CX\) is \emph{rigid} in \(\CC\) if 
\(\tensor*[]{\Ext}{_{\CC}^{1}}(\CX,\CX) 
	= \CC(\CX,\sus\CX)
	= 0
\); see \cite[Sec.\ 2.4]{Beligiannis-rigid-objects-triangulated-subfactors-and-abelian-localizations}. 
If \(A\) is an object in \(\CC\), then a \emph{right \(\CX\)-approximation of A} is a morphism \(\xi\colon X \to A\) in \(\CC\), such that \(X\in\CX\) and for any \(X'\in\CX\) the induced homomorphism 
\(
\CC(X',\xi)\colon \CC(X',X) \to \CC(X',A)
\)
is surjective. 
If for each \(A\in\CC\) there is a right \(\CX\)-approximation \(X\to A\), then \(\CX\) is said to be \emph{contravariantly finite}. 
See \cite[pp.\ 113--114]{AuslanderReiten-apps-of-contravariantly-finite-subcats}.

Suppose \(\CY,\CZ\sse\CC\) are additive subcategories. Then \(\add(\CY * \CZ)\) denotes the closure under direct summands of 
\[
\CY * \CZ \deff 
\Set{ C\in\CC | \text{there exists a triangle } Y\to C \to Z \to \sus Y \text{ with } Y\in\CY, Z\in\CZ},
\] 
considered as a full subcategory of \(\CC\). 

\begin{rem}
We note that it is frequently the case that \(\CY * \CZ = \add(\CY * \CZ)\) by \cite[Prop.\ 2.1(1)]{IyamaYoshino-mutation-in-tri-cats-rigid-CM-mods}. 
In particular, if \(\CX\) is an additive, 
rigid subcategory of a Krull--Schmidt triangulated category \(\CC\), then \(\CX * \sus \CX\) is closed under direct summands. 
Note also that we always have \(\CY,\CZ \sse \add(\CY * \CZ)\). 
\end{rem}

Now we make several observations related to the subcategory \(\CX * \sus \CX \) of \(\CC\). 
\begin{enumerate}[(\roman*)]

	\item \label{item:3iii} Up to isomorphism, each \(L\in\rmod{\CX}\) has the form \(\tensor*[]{H}{_{\CX}}(Z)\), where \(Z\in\CC\) permits a triangle 
	\begin{equation}\label{eqn:presentation-of-A}
	\begin{tikzcd}
		\tensor[]{X}{_{1}} \arrow{r}{\tensor[]{\xi}{_{1}}}& \tensor[]{X}{_{0}}\arrow{r} & Z\arrow{r} & \sus \tensor[]{X}{_{1}}
	\end{tikzcd}
	\end{equation}
	in \(\CC\) with \(\tensor[]{X}{_{i}}\in\CX\) and, in particular, \(Z\in\CX * \sus\CX\).
	Indeed, by definition of \(\rmod{\CX}\), the functor \(L\) fits into an exact sequence 
	\begin{center}
	\(
	\begin{tikzcd}
	\CX(-, \tensor[]{X}{_{1}}) \arrow{r}{\tensor[]{\Xi}{_{1}}}& \CX(-,\tensor[]{X}{_{0}})\arrow{r} & L(-)\arrow{r} & 0
	\end{tikzcd}
	\)
	\end{center}
	in \(\rmod{\CX}\), where \(\tensor[]{X}{_{i}}\in\CX\); see Subsection~\ref{subsec:conventions-notation}\ref{item:functor-categories}.
	By Yoneda's Lemma, we have \(\tensor[]{\Xi}{_{1}} = \CX(-,\tensor[]{\xi}{_{1}})\) for some morphism \(\tensor[]{\xi}{_{1}}\colon \tensor[]{X}{_{1}}\to \tensor[]{X}{_{0}}\) in \(\CC\). Completing \(\xi\) to a distinguished triangle~\eqref{eqn:presentation-of-A} and applying the cohomological functor \(\tensor*[]{H}{_{\CX}}\) (see Subsection~\ref{subsec:conventions-notation}\ref{item:H-functor-perp-category}) yields an exact sequence 
	\begin{center}
	\(
	\begin{tikzcd}[column sep=1.5cm]
	\CX(-, \tensor[]{X}{_{1}}) \arrow{r}{\CX(-,\tensor[]{\xi}{_{1}})}& \CX(-,\tensor[]{X}{_{0}})\arrow{r} & \restr{\CC(-,Z)}{\CX}\arrow{r} & 0,
	\end{tikzcd}
	\) 
	\end{center}
	whence \(L(-) \iso \restr{\CC(-,Z)}{\CX} = \tensor*[]{H}{_{\CX}}(Z)\). This argument is well-known; see, for example, \cite[Sec.\ 4.4]{Beligiannis-rigid-objects-triangulated-subfactors-and-abelian-localizations}, or \cite[proof of Lem.\ 4.3]{BuanMarsh-BM1}.
			
	\item \label{item:3iv} For each \(A\in\CC\), there is a triangle 
	\(\begin{tikzcd}[column sep=0.5cm]
	X \arrow{r}{\xi}& A \arrow{r}& Y \arrow{r} & \sus X
	\end{tikzcd}\)
	in \(\CC\), such that 
	\(X\in\CX\), 
	\(Y\in\tensor[]{\CX}{^{\perp_{0}}}\) and 
	\(\xi\) is a right \(\CX\)-approximation. 		
	See \cite[Sec.\ 2.1, Lem.\ 2.2(i)]{Beligiannis-rigid-objects-triangulated-subfactors-and-abelian-localizations}, also \cite[Lem.\ 1.2]{BuanMarsh-BM1}.
	
		\item \label{item:3ii} Given \(A\in\CC\), there is a triangle 
		\(\begin{tikzcd}[column sep=0.5cm]
			Y \arrow{r}& Z\arrow{r}{\zeta} & A\arrow{r} & \sus Y
		\end{tikzcd}\)
	in \(\CC\), such that 
	\(Y\in\tensor[]{\CX}{^{\perp_{0}}}\), and 
	\(Z\in\add(\CX * \sus \CX)\) and \(\zeta\) is a right \(\add(\CX * \sus\CX)\)-approximation. 
	Furthermore, for any \(Z'\in \add(\CX * \sus\CX)\), the homomorphism 
	\(\CC/[\tensor[]{\CX}{^{\perp_{0}}}](Z',\zeta)\colon \CC/[\tensor[]{\CX}{^{\perp_{0}}}](Z',Z) \to \CC/[\tensor[]{\CX}{^{\perp_{0}}}](Z',A)\) 
	is bijective. 
	See \cite[Eq.\ (4.2), Rem.\ 4.3, Lem.\ 4.4 and its proof]{Beligiannis-rigid-objects-triangulated-subfactors-and-abelian-localizations}, also \cite[Lem.\ 3.3]{BuanMarsh-BM1}.

\end{enumerate}


\subsection{Localisation in triangulated categories}
\label{subsec:localisation}

We also make use of localisation theory in triangulated categories. 
Again, we follow Beligiannis \cite{Beligiannis-rigid-objects-triangulated-subfactors-and-abelian-localizations}. The setup here is the same as for Subsection~\ref{subsec:approximation}.

\begin{setup}
\label{setup:localisation}

Let \(\CC\) be a skeletally small, idempotent complete, triangulated category. 
Suppose \(\CX\sse\CC\) is an additive (see Subsection~\ref{subsec:conventions-notation}\ref{item:def-additive-category}), contravariantly finite, rigid subcategory.

\end{setup}

Following \cite{Beligiannis-rigid-objects-triangulated-subfactors-and-abelian-localizations}, we denote by 
\(\ul{(-)}\colon \CC\to \CC/[\tensor[]{\CX}{^{\perp_{0}}}]\) the canonical quotient functor. 
Since \(\tensor[]{\CX}{^{\perp_{0}}} = \Ker \tensor*[]{H}{_{\CX}}\), there is a commutative diagram as follows.
\begin{equation}
\label{eqn:H-X-comm-diag}
\begin{tikzcd}
	\CC \arrow{r}{\ul{(-)}}\arrow{d}[swap]{\tensor*[]{H}{_{\CX}}}& \CC/[\tensor[]{\CX}{^{\perp_{0}}}] \arrow{dl}{\tensor[]{\ul{H}}{_{\CX}}}\\
	\rmod{\CX}&
\end{tikzcd}
\end{equation}
Beligiannis proved that the class \(\SR\deff\tensor[]{\SR}{_{\CX}}\) of regular morphisms in \(\CC/[\tensor[]{\CX}{^{\perp_{0}}}]\) admits a \emph{calculus of left and right fractions} (in the sense of Gabriel--Zisman \cite[Sec.\ I.2]{GabrielZisman-calc-of-fractions}); see \cite[Thm.\ 4.6]{Beligiannis-rigid-objects-triangulated-subfactors-and-abelian-localizations}. 
Recall that a morphism is \emph{regular} if it is simultaneously a monomorphism and an epimorphism; see \cite[p.\ 163]{Rump-almost-abelian-cats}.

\begin{rem}
Generalising work of Buan--Marsh \cite{BuanMarsh-BM2}, Beligiannis showed that 
if \(\tensor[]{\CX}{^{\perp_{0}}}\) is contravariantly finite in \(\CC\), 
then the additive category \(\CC/[\tensor[]{\CX}{^{\perp_{0}}}]\) has the structure of an integral category; see \cite[Prop.\ 4.12]{Beligiannis-rigid-objects-triangulated-subfactors-and-abelian-localizations}. Under other assumptions, this quotient category has more structure, e.g.\ 
quasi-abelian 
	\cite[Thm.\ 5.5]{Shah-quasi-abelian-hearts-of-twin-cotorsion-pairs-on-triangulated-cats}, 
abelian 
	\cite[Thm.\ 2.2]{BuanMarshReiten-cluster-tilted-algebras}, 
	\cite[Thm.\ 6.4]{Nakaoka-cotorsion-pairs-I}. 
An analogue in the context of higher homological algebra has also attracted interest \cite[Thm.\ 0.6]{JacobsenJorgensen-d-abelian-quotients-of-dplus2-angulated-categories}.
\end{rem}

The next result will often be useful for us; it follows from \cite[Lem.\ 4.1]{Beligiannis-rigid-objects-triangulated-subfactors-and-abelian-localizations} (see also \cite[Lem.\ 3.3]{BuanMarsh-BM2}).

\begin{lem}
\label{lem:Beligiannis-Lem-4-1-consequences}

Let 
\(\begin{tikzcd}[column sep=0.5cm]
A \arrow{r}{\alpha}& B \arrow{r}{\beta}& C \arrow{r}{\gamma}& \sus C
\end{tikzcd}\)
be a triangle in \(\CC\). 
Then the following hold. 
\begin{enumerate}[\textup{(\roman*)}]

	\item\label{item:Bel-Lem-4-1-i-zero} \(\ul{\beta} = 0\) in \(\CC/[\tensor[]{\CX}{^{\perp_{0}}}]\) 
				\(\iff\) \(\tensor[]{\ul{H}}{_{\CX}}(\ul{\beta}) = \tensor*[]{H}{_{\CX}}(\beta) = 0\) in \(\rmod{\CX}\). 
	
	\item\label{item:Bel-Lem-4-1-ii-epi} \(\ul{\beta}\) is an epimorphism in \(\CC/[\tensor[]{\CX}{^{\perp_{0}}}]\) 
				\(\iff\) \(\tensor[]{\ul{H}}{_{\CX}}(\ul{\beta}) = \tensor*[]{H}{_{\CX}}(\beta)\) is an epimorphism in \(\rmod{\CX}\) 
				\(\iff\) \(\ul{\gamma} = 0\) in \(\CC/[\tensor[]{\CX}{^{\perp_{0}}}]\). 
	
	\item\label{item:Bel-Lem-4-1-iii-mono} \(\ul{\beta}\) is a monomorphism in \(\CC/[\tensor[]{\CX}{^{\perp_{0}}}]\) 
				\(\iff\) \(\tensor[]{\ul{H}}{_{\CX}}(\ul{\beta}) = \tensor*[]{H}{_{\CX}}(\beta)\) is a monomorphism in \(\rmod{\CX}\) 
				\(\iff\) \(\ul{\alpha} = 0\) in \(\CC/[\tensor[]{\CX}{^{\perp_{0}}}]\). 
	
	\item\label{item:Bel-Lem-4-1-iv-iso} \(\ul{\beta}\in\SR\) 
				\(\iff\) \(\tensor[]{\ul{H}}{_{\CX}}(\ul{\beta}) = \tensor*[]{H}{_{\CX}}(\beta)\) is an isomorphism in \(\rmod{\CX}\) 
				\(\iff\) \(\ul{\alpha} = 0 = \ul{\gamma}\) in \(\CC/[\tensor[]{\CX}{^{\perp_{0}}}]\). 

\end{enumerate}
\end{lem}

Since \(\SR\) admits a calculus of left and right fractions, there is a commutative diagram
\begin{equation}
\label{eqn:GZ-localisation-diagram}
\begin{tikzcd}[row sep=1.2cm]
	\CC/[\tensor[]{\CX}{^{\perp_{0}}}] \arrow{r}{\tensor[]{\ul{H}}{_{\CX}}}\arrow{d}[swap]{\tensor[]{L}{_{\SR}}}& \rmod{\CX}\\
	(\CC/[\tensor[]{\CX}{^{\perp_{0}}}]\tensor[]{)}{_{\SR}}\arrow{ur}{\simeq}[swap]{\tensor[]{\wt{H}}{_{\CX}}},&
\end{tikzcd}
\end{equation}
where \(\tensor[]{L}{_{\SR}}\) is the canonical localisation functor and \(\tensor[]{\wt{H}}{_{\CX}}\) is the functor induced from the fact that \(\tensor[]{\ul{H}}{_{\CX}}\) sends morphisms in \(\SR\) to isomorphisms in \(\rmod{\CX}\). 
Moreover, it follows from \cite[Thm.\ 4.6 and its proof]{Beligiannis-rigid-objects-triangulated-subfactors-and-abelian-localizations} that \(\tensor[]{\wt{H}}{_{\CX}}\) is an equivalence of categories (although this is not the equivalence in the statement of \cite[Thm.\ 4.6]{Beligiannis-rigid-objects-triangulated-subfactors-and-abelian-localizations} itself).


\subsection{Extriangulated categories from triangulated categories}
\label{subsec:extriangulated-categories}

In the pursuit of unifying the theory of exact and triangulated categories, Nakaoka--Palu defined \emph{extriangulated categories}; see \cite[Def.\ 2.12]{NakaokaPalu-extriangulated-categories-hovey-twin-cotorsion-pairs-and-model-structures}. 
This attempt has been very successful, with many results that hold individually for exact and triangulated categories having been generalised to the extriangulated framework; see e.g.\ 
\cite{HassounShah-integral-and-quasi-abelian-hearts-of-twin-cotorsion-pairs-on-extriangulated-categories}, 
\cite{IyamaNakaokaPalu-Auslander-Reiten-theory-in-extriangulated-categories}, 
\cite{LiuYNakaoka-hearts-of-twin-cotorsion-pairs-on-extriangulated-categories},
\cite{Msapato-the-karoubi-envelope-and-weak-idempotent-completion-of-an-extriangulated-category}, 
\cite{NakaokaOgawaSakai-localization-of-extriangulated-categories}, 
amongst many others.

In this section, we will briefly recall the definition of an extriangulated category. However, since we will not be dealing with the axioms of extriangulated categories, we do not provide all the details here. For more details we refer the reader to \cite[Sec.\ 2]{NakaokaPalu-extriangulated-categories-hovey-twin-cotorsion-pairs-and-model-structures}. 
The extriangulated categories we consider in this article arise from triangulated categories, and we explain this relationship in Example~\ref{example:triangulated-category-is-extriangulated}.

\begin{defn}

An \emph{extriangulated category} is a triple \((\CC,\BE,\fs)\), satisfying the following.
\begin{enumerate}[(\roman*)]
	\item \(\CC\) is an additive category.
	\item \(\BE\colon \tensor[]{\CC}{^{\op}}\times \CC \to \Ab\) is a biadditive functor.
	\item \(\fs\) is an assignment that, for each \(A,C\in\CC\), assigns to each \emph{extension} \(\delta\in\BE(C,A)\) an 
	\emph{equivalence class}
	\(
	[ \begin{tikzcd}[column sep=0.5cm]
	A \arrow{r} & B \arrow{r} & C
	\end{tikzcd} ]
	\) of a \(3\)-term complex in \(\CC\),
	where two complexes 
	\(
	\begin{tikzcd}[column sep=0.5cm]
		\hspace{-5pt}A \arrow{r} & B \arrow{r} & C
	\end{tikzcd} 
	\) 
	and 
	\(
	\begin{tikzcd}[column sep=0.5cm]
	A \arrow{r} & B' \arrow{r} & C
	\end{tikzcd} 
	\) 
	are \emph{equivalent} if there is an isomorphism \(b\colon B\to B'\) and a commutative diagram 
	\[
	\begin{tikzcd}
	A \arrow{r}\arrow[equals]{d} & B \arrow{r} \arrow{d}{\iso}[swap]{b}& C \arrow[equals]{d}\\
	A \arrow{r} & B' \arrow{r} & C.
	\end{tikzcd}
	\]
	\item \(\fs\) is an \emph{additive realisation} of \(\BE\) in the sense of \cite[Defs.\ 2.9, 2.10]{NakaokaPalu-extriangulated-categories-hovey-twin-cotorsion-pairs-and-model-structures}.
	\item \(\CC\), \(\BE\) and \(\fs\) satisfy the axioms in \cite[Def.\ 2.12]{NakaokaPalu-extriangulated-categories-hovey-twin-cotorsion-pairs-and-model-structures}.
\end{enumerate}
\end{defn}

If \((\CC,\BE,\fs)\) is an extriangulated category with \(\delta\in\BE(C,A)\) 
and 
\[
\fs(\delta) 
	= [ \begin{tikzcd}
			A \arrow{r} & B \arrow{r} & C
		 \end{tikzcd} ],
\]
then we call 
\(
X^{\combul}\colon
	\begin{tikzcd}[column sep=0.5cm]
		A \arrow{r} & B \arrow{r} & C
	\end{tikzcd}
\)
an \emph{\(\fs\)-conflation}. 
Furthermore, in this case, \(\lan X^{\combul}, \delta \ran\) is called an \emph{\(\BE\)-triangle} (or simply an \emph{extriangle} if there is no confusion) and we usually depict this with the diagram
\(\begin{tikzcd}[column sep=0.5cm]
A \arrow{r} & B \arrow{r} & C	 \arrow[dashed]{r}{\delta} & {}.
\end{tikzcd}\) 
See \cite[Defs.\ 2.15, 2.19]{NakaokaPalu-extriangulated-categories-hovey-twin-cotorsion-pairs-and-model-structures}.

Each triangulated category is an example of an extriangulated category. The following explains how.

\begin{example}
\label{example:triangulated-category-is-extriangulated}

Suppose \(\CC\) is a triangulated category with suspension functor \(\sus\). 
This data allows us to define a biadditive functor 
\(\BE\colon \tensor[]{\CC}{^{\op}}\times\CC\to\Ab\) 
and an additive realisation \(\fs\) of \(\BE\) as follows. 
We define \(\BE\) on objects by 
\[
\BE(C,A)\deff \CC(C,\sus A).
\]
For morphisms \(\gamma\colon C'\to C\) and \(\alpha\colon A\to A'\), 
the morphism 
\(
\BE(\gamma,\alpha)\colon \BE(C,A) \to \BE(C',A')
\)
 is given by 
\[
\BE(\gamma,\alpha)(\delta)\deff (\sus \alpha)\circ \delta\circ \gamma.
\]

Now let \(\delta\in\BE(C,A)\) be an extension, i.e.\ \(\delta\colon C \to \sus A\) is a morphism in \(\CC\), and so fits into a triangle of the form 
\begin{equation}
\label{eqn:triangle-for-delta}
\begin{tikzcd}
A \arrow{r}{\alpha} & B \arrow{r}{\beta} & C	 \arrow{r}{\delta} & \sus A.
\end{tikzcd}
\end{equation}
In order to define a realisation \(\fs\) of \(\BE\), we first need to assign an equivalence class of an \(3\)-term complex to \(\delta\) that has first term \(A\) and last term \(C\). There is a canonical choice given a triangle \eqref{eqn:triangle-for-delta}: we put 
\[\fs(\delta) = 
	[ \begin{tikzcd}
	A \arrow{r}{\alpha} & B \arrow{r}{\beta} & C
	\end{tikzcd} ]. 
\]
Nakaoka--Palu showed that \(\fs\) defined in this way gives an additive realisation of \(\BE\) 
and, moreover, \((\CC,\BE,\fs)\) is an extriangulated category; 
see \cite[Prop.\ 3.22]{NakaokaPalu-extriangulated-categories-hovey-twin-cotorsion-pairs-and-model-structures}. 
\end{example}


We will need the theory of relative exangulated structures in the sense of \cite[Sec.\ 3.2]{HerschendLiuNakaoka-n-exangulated-categories-I-definitions-and-fundamental-properties}. 
Note that a category is extriangulated if and only if it is \(1\)-exangulated in the sense of Herschend--Liu--Nakaoka; see \cite[Prop.\ 4.3]{HerschendLiuNakaoka-n-exangulated-categories-I-definitions-and-fundamental-properties}. 
However, since we will need relative exangulated structures solely for extriangulated categories arising from triangulated categories, we provide a short exposition only in this context. 

Suppose \(\CC\) is a triangulated category with suspension functor \(\sus\), and denote the corresponding extriangulated category by \((\CC,\BE,\fs)\); see Example~\ref{example:triangulated-category-is-extriangulated}.  
Given a full subcategory \(\DD\) of \(\CC\), 
we recall how one obtains a \emph{relative} extriangulated category 
\((\CC,\tensor[]{\BE}{_{\DD}},\tensor[]{\fs}{_{\DD}})\). 
First, recall from Subsection~\ref{subsec:conventions-notation}\ref{item:H-functor-perp-category} that the functor 
\(\tensor[]{H}{_{\DD}}\colon \CC \to \rmod{\DD}\) is given by \(\tensor[]{H}{_{\DD}}(C) = \restr{\CC(-,C)}{\DD}\).  
Define the subfunctor \(\tensor[]{\BE}{_{\DD}} \colon \tensor[]{\CC}{^{\op}}\times\CC\to\Ab\) of \(\BE\) on objects as follows:
\begin{align}
	\tensor[]{\BE}{_{\DD}}(C,A)  
		&\deff \Set{ \delta\in\BE(C,A) | \tensor[]{H}{_{\DD}}(\delta) = 0 } \label{eqn:definition-of-relative-bifunctor}\\[5pt]
		& \hspace{2.5pt}= \Set{ \delta\in\BE(C,A) | 
				\begin{aligned}
					&\hspace{-5pt}\begin{tikzcd}[column sep=0.5cm, ampersand replacement=\&]
					D \arrow{r}{\xi}\& C \arrow{r}{\delta}\& \sus A
					\end{tikzcd}
					\text{ is zero for}\\[-7pt]
					&\text{all } \xi\colon D \to C \text{ with } D\in\DD
				\end{aligned}}.\nonumber
\end{align}
An additive realisation of \(\tensor[]{\BE}{_{\DD}}\) is given by the restriction 
\(\tensor[]{\fs}{_{\DD}}\)
of \(\fs\) to \(\tensor[]{\BE}{_{\DD}}\). 
Then \((\CC,\tensor[]{\BE}{_{\DD}},\tensor[]{\fs}{_{\DD}})\) is an extriangulated category. 
Furthermore, in this relative extriangulated category, all objects in $\DD$ become \emph{$\tensor[]{\BE}{_{\DD}}$-projective} in the sense of \cite[Exam.\ 5.9]{Bennett-TennenhausHauglandSandoyShah-the-category-of-extensions-and-a-characterisation-of-n-exangulated-functors}
(see also \cite[Def.\ 3.23]{NakaokaPalu-extriangulated-categories-hovey-twin-cotorsion-pairs-and-model-structures}).
See \cite[Def.\ 2.11, Prop.\ 3.16, Def.\ 3.18, Prop.\ 3.19]{HerschendLiuNakaoka-n-exangulated-categories-I-definitions-and-fundamental-properties}.

As noted in \cite[Thm.\ 2.12]{JorgensenShah-grothendieck-groups-of-d-exangulated-categories-and-a-modified-CC-map}, the extriangulated category \((\CC,\tensor[]{\BE}{_{\DD}},\tensor[]{\fs}{_{\DD}})\) is an \emph{extriangulated subcategory} in the sense of \cite[Def.\ 3.7]{Haugland-the-grothendieck-group-of-an-n-exangulated-category}. In particular, the pair 
\((\idfunc{\CC},\Delta) \colon (\CC,\tensor[]{\BE}{_{\DD}},\tensor[]{\fs}{_{\DD}})\to (\CC,\BE,\fs) \) 
is an extriangulated functor in the sense of \cite[Def.\ 2.32]{Bennett-TennenhausShah-transport-of-structure-in-higher-homological-algebra}, where \(\idfunc{\CC}\) is the identity functor on \(\CC\) and \(\Delta \colon \tensor[]{\BE}{_{\DD}} \Rightarrow \BE\) is the canonical natural transformation given by inclusion.


\section{The index with respect to a rigid subcategory}
\label{sec:relative-index-wrt-rigid-subcategory}

In this section, we prove Theorem~\ref{thmx:X-index-additive-on-triangles-with-error-term} from Section~\ref{sec:introduction}. As such, 
\(\tensor[]{K}{_{0}}\)-groups of skeletally small extriangulated categories play a central role here. 
Before we restrict to Setup~\ref{setup:1}, we provide some definitions in more generality. 
The following definition is equivalent to that of the Grothendieck group of an \(n\)-exangulated category in the sense of Haugland \cite{Haugland-the-grothendieck-group-of-an-n-exangulated-category} with \(n=1\); see \cite[Rem.\ 2.3]{Fedele-grothendieck-groups-of-triangulated-categories-via-cluster-tilting-subcategories}, \cite[Def.\ 2.18]{JorgensenShah-grothendieck-groups-of-d-exangulated-categories-and-a-modified-CC-map}.

\begin{defn}
\label{def:Grothendieck-group-of-extriangulated-category}
Suppose \((\CC,\BE,\fs)\) is a skeletally small extriangulated category. 
The \emph{Grothendieck group} of 
\((\CC,\BE,\fs)\) is defined to be 
\[
\tensor[]{K}{_{0}}(\CC,\BE,\fs) 
	\deff 
		\tensor*[]{K}{_{0}^{\sp}}(\CC)  /  \CI,
\] 
where 
\(
\CI
	\deff 
\left.\big\lan\,
	[A\tensor*[]{]}{_{\CC}^{\sp}} - [B\tensor*[]{]}{_{\CC}^{\sp}} + [C\tensor*[]{]}{_{\CC}^{\sp}} \,\middle|\, \hspace{-5pt}\begin{tikzcd}[column sep=0.5cm, ampersand replacement=\&]
			A \arrow{r}\& B \arrow{r}\& C
			\end{tikzcd} 
			\text{ is an } \fs\text{-conflation in } (\CC,\BE,\fs) \,\big\ran\right.
\).
\end{defn}

\begin{rem}
\label{rem:grothendieck-group-of-relative-structure}
Suppose \((\CC,\BE,\fs)\) is a skeletally small extriangulated category and consider its Grothendieck group 
\(
\tensor[]{K}{_{0}}(\CC,\BE,\fs) 
	= \tensor*[]{K}{_{0}^{\sp}}(\CC)  /  \CI,
\) 
as given in Definition~\ref{def:Grothendieck-group-of-extriangulated-category}.
Let \(\DD\) be a full subcategory of \(\CC\). 
Then the Grothendieck group
of 
\((\CC,\tensor[]{\BE}{_{\DD}},\tensor[]{\fs}{_{\DD}})\) 
is by definition
\(
\tensor[]{K}{_{0}}(\CC,\tensor[]{\BE}{_{\DD}},\tensor[]{\fs}{_{\DD}})
	 = \tensor*[]{K}{_{0}^{\sp}}(\CC)/\tensor*[]{\CI}{_{\DD}},
\)
where 
\[
\tensor*[]{\CI}{_{\DD}}  \deff 
\Braket{
	[A\tensor*[]{]}{_{\CC}^{\sp}} - [B\tensor*[]{]}{_{\CC}^{\sp}} + [C\tensor*[]{]}{_{\CC}^{\sp}} | \hspace{-5pt}\begin{tikzcd}[column sep=0.5cm, ampersand replacement=\&]
			A \arrow{r}\& B \arrow{r}\& C
			\end{tikzcd} 
			\text{ is an } \tensor[]{\fs}{_{\DD}}\text{-conflation in } (\CC,\tensor[]{\BE}{_{\DD}},\tensor[]{\fs}{_{\DD}})
}.
\]
Thus, the generators of \(\tensor*[]{\CI}{_{\DD}}\) arise from extriangles 
\(
\begin{tikzcd}[column sep=0.5cm]
	A \arrow{r} & B \arrow{r} & C \arrow[dashed]{r}{\delta} & {}
\end{tikzcd} 
\)
in \((\CC,\BE,\fs)\) for which 
\(
\tensor[]{H}{_{\DD}}(\delta) = 0
\). 
In particular, 
\(\tensor*[]{\CI}{_{\DD}}\) is contained in \(\CI\). 
\end{rem}

Our first result (see Proposition~\ref{lemma7} below) concerning the index is that it is \emph{additive on extriangles} in a certain extriangulated category in the following sense.

\begin{defn}
\label{def:additive-function-on-extriangulated-category}
Let \((\CC,\BE,\fs)\) be an extriangulated category, and let \(\obj(\CC)\) denote the class of objects in \(\CC\). 
Let \(G\) be any abelian group. 
An assignment \(f\colon \obj(\CC)\to G\) is \emph{additive on extriangles in \((\CC,\BE,\fs)\)} if, for each extriangle 
\(\begin{tikzcd}[column sep=0.5cm]
A \arrow{r}& B \arrow{r}& C \arrow[dashed]{r}{\delta} & {},
\end{tikzcd}\)
we have 
\(
f(A) - f(B) + f(C) = 0
\).
\end{defn}

The prototypical example of an additive assignment on extriangles in an extriangulated category is the length function on short exact sequences in a module category.


\begin{setup}
\label{setup:1}

Let \(\CC\) be a skeletally small, idempotent complete, triangulated category. 
Suppose \(\CX\sse\CC\) is an additive, contravariantly finite, rigid subcategory. 
Recall that by an \emph{additive} subcategory we assume closure under isomorphisms, direct sums and direct summands (see Subsection~\ref{subsec:conventions-notation}\ref{item:def-additive-category}).

\end{setup}

In particular, Setup~\ref{setup:1} is the same as Setups~\ref{setup:approximation} and \ref{setup:localisation}.

Our goal for the remainder of this section is to establish Theorem~\ref{thmx:X-index-additive-on-triangles-with-error-term}. 
We begin by defining the index on \(\CC\) with respect to \(\CX\) (see Definition~\ref{def:relative-index-wrt-X}). 
We show this index is 
additive on extriangles in \((\CC,\tensor[]{\BE}{_{\CX}},\tensor[]{\fs}{_{\CX}})\) (see Proposition~\ref{lemma7}), 
and 
additive on triangles in \(\CC\) up to an error term (see Theorem~\ref{thm:X-index-additive-on-triangles-with-error-term}). 
In order to prove this, we develop several analogues of results from \cite{Jorgensen-tropical-friezes-and-the-index-in-higher-homological-algebra}. 
We omit the proofs that are entirely similar, but state the needed modifications.

\begin{defn}
\label{def:relative-index-wrt-X}

If \(C\in\CC\), then denote by \(\newindx{\CX}(C) \deff [C\tensor[]{]}{_{\CX}}\) the class of \(C\) in the Grothendieck group 
\(\tensor[]{K}{_{0}}(\CC,\tensor[]{\BE}{_{\CX}},\tensor[]{\fs}{_{\CX}})\).  
We call  \(\newindx{\CX}\) the \emph{index with respect to \(\CX\)}.

\end{defn}

The index with respect to \(\CX\) also induces a well-defined group homomorphism 
\(\tensor*[]{K}{_{0}^{\sp}}(\CC) \to \tensor[]{K}{_{0}}(\CC,\tensor[]{\BE}{_{\CX}},\tensor[]{\fs}{_{\CX}})\), 
which we also denote by \(\newindx{\CX}\) by abuse of notation.

\begin{prop}
\label{lemma7}
If 
\(\begin{tikzcd}[column sep=0.5cm]
	A \arrow{r} & B \arrow{r}& C \arrow{r}{\delta} & \sus A
\end{tikzcd}\)
is a triangle in \(\CC\) with \(\tensor*[]{H}{_{\CX}}(\delta) = 0\), then 
\(\newindx{\CX}(B) = \newindx{\CX}(A) + \newindx{\CX}(C)\). 
In particular, \(\newindx{\CX}\) is additive on extriangles in \((\CC,\tensor[]{\BE}{_{\CX}},\tensor[]{\fs}{_{\CX}})\).

\end{prop}

\begin{proof}

Since \(\tensor*[]{H}{_{\CX}}(\delta) = 0\), we have that \(\delta\) lies in \(\tensor[]{\BE}{_{\CX}}(C,A)\) (see \eqref{eqn:definition-of-relative-bifunctor} in Subsection~\ref{subsec:extriangulated-categories}), so 
\(\tensor[]{\fs}{_{\CX}}(\delta) = 
	[
	\begin{tikzcd}[column sep=0.5cm]
		A \arrow{r} & B \arrow{r}& C
	\end{tikzcd}
	]
\). 
Thus,
\[
\newindx{\CX}(A) - \newindx{\CX}(B) + \newindx{\CX}(C) = [A\tensor[]{]}{_{\CX}} - [B\tensor[]{]}{_{\CX}} + [C\tensor[]{]}{_{\CX}}= 0
\] in \(\tensor[]{K}{_{0}}(\CC,\tensor[]{\BE}{_{\CX}},\tensor[]{\fs}{_{\CX}})\) and the first claim follows. 

Recall that the extriangles 
\(\begin{tikzcd}[column sep=0.5cm]
	A \arrow{r} & B \arrow{r}& C \arrow[dashed]{r}{\delta} & {}
\end{tikzcd}\) 
in \((\CC,\tensor[]{\BE}{_{\CX}},\tensor[]{\fs}{_{\CX}})\) 
satisfy 
\(\tensor*[]{H}{_{\CX}}(\delta) = 0\) (see \eqref{eqn:definition-of-relative-bifunctor} in Subsection~\ref{subsec:extriangulated-categories}). 
Thus, the equation above shows that \(\newindx{\CX}\) is indeed additive on extriangles in \((\CC,\tensor[]{\BE}{_{\CX}},\tensor[]{\fs}{_{\CX}})\) 
in the sense of Definition~\ref{def:additive-function-on-extriangulated-category}.
\end{proof}

We see later that \(\newindx{\CX}\) is additive on triangles in \(\CC\) up to an error term; see Theorem~\ref{thm:X-index-additive-on-triangles-with-error-term}. 
How far this index is from being additive is measured by a certain abelian group homomorphism 
\(\tensor[]{\theta}{_{\CX}} \); 
see Proposition~\ref{prop:definition-of-theta-sub-CX}. 
For this result we need several preliminary lemmas.

\begin{lem}
\label{lemma8}

For \(C,D\in\CC\), the following are equivalent.
\begin{enumerate}[\textup{(\roman*)}]

	\item \(\tensor*[]{H}{_{\CX}}(C) \iso \tensor*[]{H}{_{\CX}}(D)\) in \(\rmod{\CX}\).
	
	\item There is a diagram
		\(\begin{tikzcd}[column sep=0.5cm]
		C \arrow{r}{\gamma}& A & \arrow{l}[swap]{\delta}D
		\end{tikzcd}\)
		in \(\CC\) with \(\ul{\gamma}, \ul{\delta}\in\SR\).
		
	\item There is a diagram
		\(\begin{tikzcd}[column sep=0.5cm]
		C & B \arrow{l}[swap]{\mu} \arrow{r}{\nu} & D
		\end{tikzcd}\)
		in \(\CC\) with \(\ul{\mu}, \ul{\nu}\in\SR\).

\end{enumerate}

\end{lem}

\begin{proof}

We only show (i)\(\iff\)(ii); the proof of (i)\(\iff\)(iii) is similar.

(ii)\(\implies\)(i): \;\;This follows from Lemma~\ref{lem:Beligiannis-Lem-4-1-consequences}\ref{item:Bel-Lem-4-1-iv-iso}. 

(i)\(\implies\)(ii): \;First, \(\tensor*[]{H}{_{\CX}}(C) \iso \tensor*[]{H}{_{\CX}}(D)\) implies \(\tensor[]{\ul{H}}{_{\CX}}(C) \iso \tensor[]{\ul{H}}{_{\CX}}(D)\) (see \eqref{eqn:H-X-comm-diag}). 
Using that \(\tensor[]{\wt{H}}{_{\CX}}\colon (\CC/[\tensor[]{\CX}{^{\perp_{0}}}]\tensor[]{)}{_{\SR}} \to \rmod{\CX}\) is an equivalence (see \eqref{eqn:GZ-localisation-diagram}), 
we see there is an isomorphism \(\phi\colon \tensor[]{L}{_{\SR}}(C) \to \tensor[]{L}{_{\SR}}(D)\) in the localisation \((\CC/[\tensor[]{\CX}{^{\perp_{0}}}]\tensor[]{)}{_{\SR}}\). Since \(\SR\) admits a calculus of left fractions in \(\CC/[\tensor[]{\CX}{^{\perp_{0}}}]\), the morphism \(\phi\) is represented by a roof/left fraction 
\(\begin{tikzcd}[column sep=0.5cm]
	C \arrow{r}{\ul{\gamma}}& A & \arrow{l}[swap]{\ul{\delta}}D
\end{tikzcd}\) 
with \(\ul{\delta}\in\SR\) (see \cite[Sec.\ I.2]{GabrielZisman-calc-of-fractions}). 
In particular, \(\tensor[]{\ul{H}}{_{\CX}}(\ul{\delta})\) is an isomorphism in \(\rmod{\CX}\) by Lemma~\ref{lem:Beligiannis-Lem-4-1-consequences}\ref{item:Bel-Lem-4-1-iv-iso}. 
Thus, \(\tensor[]{\wt{H}}{_{\CX}}(\phi) = (\tensor[]{\ul{H}}{_{\CX}}(\ul{\delta})\tensor[]{)}{^{-1}}\tensor[]{\ul{H}}{_{\CX}}(\ul{\gamma})\) implies 
\(\tensor[]{\ul{H}}{_{\CX}}(\ul{\gamma}) = \tensor[]{\ul{H}}{_{\CX}}(\ul{\delta}) \tensor[]{\wt{H}}{_{\CX}}(\phi)\) 
is a composition of isomorphisms, and hence an isomorphism itself. So \(\ul{\gamma}\in\SR\) by 
Lemma~\ref{lem:Beligiannis-Lem-4-1-consequences}\ref{item:Bel-Lem-4-1-iv-iso}. 
\end{proof}

The following result is an analogue of \cite[Lem.\ 4.2]{Jorgensen-tropical-friezes-and-the-index-in-higher-homological-algebra}.

\begin{lem}
\label{lemma9}

Suppose \(\tensor*[]{H}{_{\CX}}(C) \iso \tensor*[]{H}{_{\CX}}(D)\) in \(\rmod{\CX}\) for some objects \(C,D\in\CC\). 
Then 
\(\newindx{\CX}(C) + \newindx{\CX}(\tensor*[]{\sus}{^{-1}}C) 
	= \newindx{\CX}(D) + \newindx{\CX}(\tensor*[]{\sus}{^{-1}}D)\). 

\end{lem}

\begin{proof}

By Lemma~\ref{lemma8}, it is enough to prove this when there is a morphism \(\psi\colon C \to D\) with \(\ul{\psi}\in\SR\), i.e.\ when \(\tensor*[]{H}{_{\CX}}(\psi)\) is an isomorphism (see Lemma~\ref{lem:Beligiannis-Lem-4-1-consequences}). 
Completing \(\psi\) to a triangle in \(\CC\) yields 
\(\begin{tikzcd}[column sep=0.5cm]
	B \arrow{r}{\chi}& C \arrow{r}{\psi}& D \arrow{r}{\phi}& \sus B.
\end{tikzcd}\) 
Rotating this triangle yields another triangle 
\(\begin{tikzcd}[column sep=1.3cm]
	\tensor*[]{\sus}{^{-1}}C \arrow{r}{-\tensor*[]{\sus}{^{-1}}\psi}& \tensor*[]{\sus}{^{-1}}D \arrow{r}{-\tensor*[]{\sus}{^{-1}}\phi}& B \arrow{r}{\chi}& C.
\end{tikzcd}\) 
Since \(\tensor*[]{H}{_{\CX}}(\psi)\) is an isomorphism, we have \(\tensor*[]{H}{_{\CX}}(\phi) = \tensor*[]{H}{_{\CX}}(\chi) = 0\) by 
Lemma~\ref{lem:Beligiannis-Lem-4-1-consequences}. 
Thus, by Proposition~\ref{lemma7} applied to these triangles, we see that
\begin{align*}
	\newindx{\CX}(B) - \newindx{\CX}(C) + \newindx{\CX}(D) &=  0,\\
	\newindx{\CX}(\tensor*[]{\sus}{^{-1}}C) - \newindx{\CX}(\tensor*[]{\sus}{^{-1}}D) + \newindx{\CX}(B) &=  0.
\end{align*}
Combining these gives the desired result.
\end{proof}

The next lemma gives a generalisation of \cite[Lem.\ 1.11(ii)]{HolmJorgensen-generalized-friezes-and-a-modified-caldero-chapoton-map-depending-on-a-rigid-object-1}, and 
it can be checked that the proof given there works for our purposes as well with the addition of the ``General Case'' we provide.
However, the proof we give below is inspired by \cite[Lem.\ 3.6, Prop.\ 3.7]{BuanMarsh-BM1}. Buan--Marsh attribute the proof of \cite[Lem.\ 3.6]{BuanMarsh-BM1} to Yann Palu. 
Recall that the subcategory \(\add(\CX*\sus\CX)\) of \(\CC\) consists of direct summands of objects \(Z\) in \(\CC\) that admit a triangle 
\(\begin{tikzcd}[column sep=0.5cm]
	\tensor[]{X}{_{0}}\arrow{r} & Z \arrow{r}& \sus \tensor[]{X}{_{1}} \arrow{r}&\sus \tensor[]{X}{_{0}}
\end{tikzcd}\) 
with \(\tensor[]{X}{_{i}}\in\CX\).

\begin{lem}
\label{lemma10}

If \(Z\in\add(\CX*\sus\CX)\) and \(C\in\CC\), then the mapping 
\[
(\tensor*[]{H}{_{\CX}}\tensor[]{)}{_{Z,C}}\colon \CC(Z,C) \to (\rmod{\CX})(\tensor*[]{H}{_{\CX}}(Z), \tensor*[]{H}{_{\CX}}(C))
\]
of \(\Hom\)-sets induced by \(\tensor*[]{H}{_{\CX}}\) is surjective. 
\end{lem}

\begin{proof}
We will show that if \(Z\in\add(\CX*\sus\CX)\) and \(C\in\CC\), then 
\[
(\tensor[]{L}{_{\SR}}\circ\ul{(-)}\tensor[]{)}{_{Z,C}}\colon \CC(Z,C) \to (\CC/[\tensor[]{\CX}{^{\perp_{0}}}]\tensor[]{)}{_{\SR}}(Z, C)
\]
is surjective. The result then follows from diagrams \eqref{eqn:H-X-comm-diag} and \eqref{eqn:GZ-localisation-diagram}, noting that \(\tensor[]{\wt{H}}{_{\CX}}\) is an equivalence.

\underline{Special Case}:\; 
$Z\in \CX * \sus \CX$, i.e.\ 
there is a triangle 
\begin{equation}\label{eqn:special-case-triangle}
\begin{tikzcd}[column sep=0.5cm]
	\tensor[]{X}{_{1}} \arrow{r}{} & \tensor[]{X}{_{0}} \arrow{r}{\alpha} & Z \arrow{r}{\beta} & \sus \tensor[]{X}{_{1}}
\end{tikzcd}
\end{equation}
in \(\CC\) with \(\tensor[]{X}{_{i}}\in\CX\). 

A morphism in \((\CC/[\tensor[]{\CX}{^{\perp_{0}}}]\tensor[]{)}{_{\SR}}(Z, C)\) is given by a left fraction 
\(\begin{tikzcd}[column sep=0.5cm]
	Z \arrow{r}{\ul{\gamma}}& A & \arrow{l}[swap]{\ul{\delta}}C
\end{tikzcd}\) 
with \(\ul{\delta}\in\SR\). 
We claim that the morphism \(\ul{\gamma}\) factors through \(\ul{\delta}\) (cf.\ \cite[Lem.\ 3.6]{BuanMarsh-BM1}). 
As \(Z\in\CX*\sus\CX\), there is a triangle
\eqref{eqn:special-case-triangle} in \(\CC\) with \(\tensor[]{X}{_{i}}\in\CX\). 
There is also a triangle
\(\begin{tikzcd}[column sep=0.5cm]
	\hspace{-5pt}B \arrow{r}{\eps} & C \arrow{r}{\delta} & A \arrow{r}{\zeta} & \sus B,
\end{tikzcd}\)
where \(\eps, \zeta \in[\tensor[]{\CX}{^{\perp_{0}}}]\) by Lemma~\ref{lem:Beligiannis-Lem-4-1-consequences} as \(\ul{\delta}\in\SR\). 
Since \(\tensor[]{X}{_{0}}\) lies in \(\CX\) and \(\zeta\) lies in the ideal \([\tensor[]{\CX}{^{\perp_{0}}}]\), the composition 
\(\zeta\gamma\alpha\) is zero. Hence, there is a morphism of triangles (given by the solid arrows in the following diagram)
\[
\begin{tikzcd}[column sep=1.3cm]
	\tensor[]{X}{_{1}} \arrow{r}{} \arrow{d}{\eta} & \tensor[]{X}{_{0}} \arrow{r}{\alpha} \arrow{d}& Z \arrow{r}{\beta} \arrow{d}[swap]{\gamma} \arrow[dotted]{dl}[swap]{\tensor[]{\eps}{_{1}}}& \sus \tensor[]{X}{_{1}} \arrow{d}{\sus \eta} \arrow[dotted]{dl}[swap]{\tensor[]{\eps}{_{2}}}\\
	B \arrow{r}{\eps} & C \arrow{r}{\delta} & A \arrow{r}{\zeta} & \sus B,
\end{tikzcd}
\]
using the axioms of a triangulated category (see \cite[(TR\(4\))]{HolmJorgensen-tri-cats-intro}). 
The morphism \(\eps\eta\) also vanishes as \(\tensor[]{X}{_{1}}\in\CX\) and  \(\eps\in[\tensor[]{\CX}{^{\perp_{0}}}]\). 
Since $\tensor[]{X}{_{1}}\in\CX$ and $\eps\in[\tensor[]{\CX}{^{\perp_{0}}}]$, 
we have that $(\sus\eps)(\sus\eta) = 0$ and so there is a morphism $\tensor[]{\eps}{_{2}}\colon \sus \tensor[]{X}{_{1}} \to A$ such that $\zeta \tensor[]{\eps}{_{2}} = \sus\eta$. 
Furthermore, the vanishing of $\zeta(\gamma - \tensor[]{\eps}{_{2}}\beta)$ implies the existence of a morphism $\tensor[]{\eps}{_{1}}\colon Z\to C$ with $\delta\tensor[]{\eps}{_{1}} = \gamma - \tensor[]{\eps}{_{2}}\beta$, or \(\gamma = \delta\tensor[]{\eps}{_{1}} + \tensor[]{\eps}{_{2}}\beta\) (see the diagram just above).
Note that \(\ul{\beta} = 0\) in \(\CC/[\tensor[]{\CX}{^{\perp_{0}}}]\) as \(\CX\) is rigid. 
Hence, we have \(\ul{\gamma} = \ul{\delta\tensor[]{\eps}{_{1}}}\) in \(\CC/[\tensor[]{\CX}{^{\perp_{0}}}]\), as claimed. 
Lastly, we note this implies 
\(
\tensor[]{L}{_{\SR}}(\ul{\tensor[]{\eps}{_{1}}}) 
	= \tensor[]{L}{_{\SR}}(\ul{\delta}\tensor[]{)}{^{-1}}\tensor[]{L}{_{\SR}}(\ul{\gamma})
\), finishing the special case.

\underline{General Case}:\; 
\(Z\in\add(\CX*\sus\CX)\). 

In this case, there is an object \(Z'\in\CC\) and a triangle 
\(\begin{tikzcd}[column sep=0.5cm]
	\tensor[]{X}{_{1}}\arrow{r} & \tensor[]{X}{_{0}}\arrow{r} & Z \oplus Z'\arrow{r}& \sus \tensor[]{X}{_{1}}
\end{tikzcd}\)  
with \(\tensor[]{X}{_{i}}\in\CX\). The canonical inclusion 
\(\iota \deff 
\begin{psmallmatrix}
	\iden{Z} \\
	0
\end{psmallmatrix}
\colon Z\to Z\oplus Z'\) 
induces the following commutative diagram
\[
\begin{tikzcd}[column sep=4cm, row sep=1.3cm]
\CC(Z\oplus Z', C) \arrow{r}{\CC(\iota, C)}\arrow[two heads]{d}[swap]{(\tensor[]{L}{_{\SR}}\circ\ul{(-)})_{Z\oplus Z',C}}& \CC(Z,C) \arrow{d}{(\tensor[]{L}{_{\SR}}\circ\ul{(-)}\tensor[]{)}{_{Z,C}}}\\
(\CC/[\tensor[]{\CX}{^{\perp_{0}}}]\tensor[]{)}{_{\SR}}(Z \oplus Z', C)
	\arrow[two heads]{r}{(\CC/[\tensor[]{\CX}{^{\perp_{0}}}]\tensor[]{)}{_{\SR}}(\tensor[]{L}{_{\SR}}(\ul{\iota}), C)}
	& (\CC/[\tensor[]{\CX}{^{\perp_{0}}}]\tensor[]{)}{_{\SR}}(Z, C).
\end{tikzcd}
\]
The left vertical map is surjective by the Special Case and the lower horizontal map is a split epimorphism, and hence the composition is surjective. It follows from commutativity that the right vertical map is also surjective.
\end{proof}

The next result is a generalisation of \cite[Lem.\ 2.7]{Jorgensen-tropical-friezes-and-the-index-in-higher-homological-algebra}. The proof given there works in our setting with minor modifications: 
use Section~\ref{subsec:approximation}\ref{item:3iii} and Lemma~\ref{lemma10} in place of \cite[Lem.\ 2.4]{Jorgensen-tropical-friezes-and-the-index-in-higher-homological-algebra}; and note 
\(\tensor*[]{H}{_{\CX}}(\sus \CX) = 0\) as \(\CX\) is rigid.

\begin{lem}
\label{lemma11}

Suppose 
\(\begin{tikzcd}[column sep=0.5cm]
	0\arrow{r}& L \arrow{r}{\iota}& M \arrow{r}{\sigma}& N \arrow{r}&0 
\end{tikzcd}\) 
is a short exact sequence in \(\rmod{\CX}\). Then, up to isomorphism, it has the form 
\[
 \begin{tikzcd}[column sep=1.2cm]
		0\arrow{r} & \tensor*[]{H}{_{\CX}}(A) \arrow{r}{\tensor*[]{H}{_{\CX}}(\alpha)}& \tensor*[]{H}{_{\CX}}(B) \arrow{r}{\tensor*[]{H}{_{\CX}}(\beta)}& \tensor*[]{H}{_{\CX}}(C) \arrow{r} & 0
\end{tikzcd}
\]
for a triangle 
\(
\begin{tikzcd}[column sep=0.5cm]
	A \arrow{r}{\alpha}& B \arrow{r}{\beta}& C \arrow{r}{\gamma}& \sus A
\end{tikzcd}
\)
in \(\CC\), with \(A,B,C\in \CX * \sus \CX\). 
\end{lem}

The next lemma is an analogue of \cite[Lem.\ 4.3]{Jorgensen-tropical-friezes-and-the-index-in-higher-homological-algebra} and its proof is similar. Thus, we omit the proof and note that applications of 
\cite[Lem.\ 2.7]{Jorgensen-tropical-friezes-and-the-index-in-higher-homological-algebra}, 
\cite[Lem.\ 4.2]{Jorgensen-tropical-friezes-and-the-index-in-higher-homological-algebra}, and 
\cite[Lem.\ 3.7]{Jorgensen-tropical-friezes-and-the-index-in-higher-homological-algebra} 
may be replaced by applications of 
Lemma~\ref{lemma11}, 
Lemma~\ref{lemma9}, and 
Proposition~\ref{lemma7}, respectively.

\begin{lem}
\label{lemma12}

For a short exact sequence 
\begin{equation}\label{eqn:ses-of-lemma12}
\begin{tikzcd}
	0\arrow{r}& \tensor*[]{H}{_{\CX}}(X) \arrow{r}& \tensor*[]{H}{_{\CX}}(Y) \arrow{r}& \tensor*[]{H}{_{\CX}}(Z) \arrow{r}&0 
\end{tikzcd}
\end{equation}
in \(\rmod{\CX}\), we have 
\[
\newindx{\CX}(Y) + \newindx{\CX}(\tensor*[]{\sus}{^{-1}}Y) = \newindx{\CX}(X) + \newindx{\CX}(\tensor*[]{\sus}{^{-1}}X) + \newindx{\CX}(Z) + \newindx{\CX}(\tensor*[]{\sus}{^{-1}}Z).
\]

\end{lem}

The failure of the index \(\newindx{\CX}\) to be additive on triangles in \(\CC\) can be measured by a homomorphism \(\tensor[]{\theta}{_{\CX}} \), which we now define.

\begin{prop}
\label{prop:definition-of-theta-sub-CX}

There is a well-defined group homomorphism 
\[
\begin{tikzcd}[column sep= 1.3cm, ampersand replacement=\&]
	\begin{aligned}[t]
		\tensor[]{K}{_{0}}(\rmod{\CX})& \\
		[\tensor*[]{H}{_{\CX}}(C)]& 
	\end{aligned} 
\arrow{r}{\tensor[]{\theta}{_{\CX}}} \arrow[maps to, yshift={-0.75cm}]{r}\&
	\begin{aligned}[t]
		&\tensor[]{K}{_{0}}(\CC,\tensor[]{\BE}{_{\CX}},\tensor[]{\fs}{_{\CX}}) \\
		 &\newindx{\CX}(C) + \newindx{\CX}(\tensor*[]{\sus}{^{-1}}C).
	\end{aligned} 
\end{tikzcd}
\]
\end{prop}

\begin{proof}

Combining Section~\ref{subsec:approximation}\ref{item:3iii} and Lemma~\ref{lemma9} shows that the assignment 
\[
\tensor*[]{H}{_{\CX}}(C)\mapsto \newindx{\CX}(C) + \newindx{\CX}(\tensor*[]{\sus}{^{-1}}C),
\]
where \(C\in\CC\), induces a map 
\(\rmod{\CX} \to \tensor[]{K}{_{0}}(\CC,\tensor[]{\BE}{_{\CX}},\tensor[]{\fs}{_{\CX}})\) which is constant on isomorphism classes. Furthermore, Lemma~\ref{lemma12} implies that it factors through \(\tensor[]{K}{_{0}}(\rmod{\CX})\). 
\end{proof}

The following lemma is a corresponding version of \cite[Lem.\ 2.6]{Jorgensen-tropical-friezes-and-the-index-in-higher-homological-algebra}. A similar proof can be used. Instead of \cite[Lem.\ 2.4]{Jorgensen-tropical-friezes-and-the-index-in-higher-homological-algebra} and \cite[Lem.\ 2.5(i)]{Jorgensen-tropical-friezes-and-the-index-in-higher-homological-algebra}, use 
Section~\ref{subsec:approximation}\ref{item:3iii} and Lemma~\ref{lemma10}, and Lemma~\ref{lemma10}, respectively. In addition, note that a morphism in \(\CC\) vanishes under \(\tensor*[]{H}{_{\CX}}\) if and only if it factors through \(\tensor[]{\CX}{^{\perp_{0}}}\) (see Lemma~\ref{lem:Beligiannis-Lem-4-1-consequences}). 

\begin{lem}\label{lemma14}

Given \(\zeta\colon Z\to \sus A\) in \(\CC\) with \(Z\in\add(\CX*\sus\CX)\), there is a commutative diagram 
\[
\begin{tikzcd}[column sep=0.5cm]
	Z \arrow{rr}{\zeta} \arrow{dr}[swap]{\sigma}& &\sus A\\
	& C \arrow{ur}[swap]{\iota}&
\end{tikzcd}
\] 
in \(\CC\) with $C\in\CX*\sus\CX$, such that 
\(\tensor*[]{H}{_{\CX}}(\sigma)\) is epic 
and \(\tensor*[]{H}{_{\CX}}(\iota)\) is monic. 
Moreover, \(\tensor*[]{H}{_{\CX}}(C)\) is isomorphic to \(\Im \tensor*[]{H}{_{\CX}}(\zeta)\) in \(\rmod{\CX}\).
\end{lem}

We now present the main result of this subsection: an additivity formula with error term for the index \(\newindx{\CX}\) with respect to the rigid subcategory \(\CX\) (cf.\ \cite[Thm.\ 4.4]{Jorgensen-tropical-friezes-and-the-index-in-higher-homological-algebra}).

\begin{thm}
\label{thm:X-index-additive-on-triangles-with-error-term}

If 
\(\begin{tikzcd}[column sep=0.5cm]
	A \arrow{r}{}& B \arrow{r}{}& C \arrow{r}{\gamma}& \sus A
\end{tikzcd}\) 
is a triangle in \(\CC\), then 
\[
\newindx{\CX}(A) - \newindx{\CX}(B) + \newindx{\CX}(C) 
	=  \tensor[]{\theta}{_{\CX}}\big([\Im \tensor*[]{H}{_{\CX}}(\gamma)]\big).
\]

\end{thm}

\begin{proof}

\underline{Special Case}:\; 
\(C\in\add(\CX*\sus\CX)\). 

One can argue as in the proof of the ``Special case'' of \cite[Thm.\ 4.4]{Jorgensen-tropical-friezes-and-the-index-in-higher-homological-algebra}, using 
Lemma~\ref{lemma14}, 
Proposition~\ref{lemma7} and 
Proposition~\ref{prop:definition-of-theta-sub-CX} in place of 
\cite[Lem.\ 2.6]{Jorgensen-tropical-friezes-and-the-index-in-higher-homological-algebra}, 
\cite[Lem.\ 3.7]{Jorgensen-tropical-friezes-and-the-index-in-higher-homological-algebra} and 
\cite[Lem.\ 4.1]{Jorgensen-tropical-friezes-and-the-index-in-higher-homological-algebra}, 
respectively.

\underline{General Case}:\; 
\(C\in\CC\) arbitrary. 

By Section~\ref{subsec:approximation}\ref{item:3ii}, there is a triangle 
\(\begin{tikzcd}[column sep=0.5cm]
	D \arrow{r}& Z\arrow{r}{\zeta} & C\arrow{r}{\delta} & \sus D,
\end{tikzcd}\) 
where \(\zeta\) is a right \(\add(\CX*\sus\CX)\)-approximation of \(C\). 
Applying the octahedral axiom (see \cite[(TR5'')]{HolmJorgensen-tri-cats-intro}) on the composition \(\gamma\zeta\) produces a commutative diagram
\[
\begin{tikzcd}
	0 \arrow{r}\arrow{d}& D \arrow[equals]{r}\arrow{d}& D \arrow{r}\arrow{d}& 0\arrow{d}\\
	A \arrow{r}\arrow[equals]{d}& Y \arrow{r}\arrow{d}{\nu}& Z \arrow{r}{\eta}\arrow{d}{\zeta}& \sus A\arrow[equals]{d}\\
	A \arrow{r}\arrow{d}& B \arrow{r}{\beta}\arrow{d}{\eps}& C\arrow{r}{\gamma}\arrow{d}{\delta}& \sus A\arrow{d}\\
	0 \arrow{r}& \sus D \arrow[equals]{r}& \sus D \arrow{r}& 0 
\end{tikzcd}
\]
in which each row and column is a triangle. 
Since \(\zeta\) is a right \(\add(\CX*\sus\CX)\)-approximation, we have that \(\tensor*[]{H}{_{\CX}}(\zeta)\) is an epimorphism. Hence, \(\tensor*[]{H}{_{\CX}}(\delta)\tensor*[]{H}{_{\CX}}(\zeta) = \tensor*[]{H}{_{\CX}}(\delta\zeta) = 0\) implies \(\tensor*[]{H}{_{\CX}}(\delta)\) is the zero map. In addition, this also gives \(\tensor*[]{H}{_{\CX}}(\eps) = \tensor*[]{H}{_{\CX}}(\delta)\tensor*[]{H}{_{\CX}}(\beta) = 0\). Thus, by Proposition~\ref{lemma7}, we have 
\begin{align}
\newindx{\CX}(D) - \newindx{\CX}(Z) + \newindx{\CX}(C)&=0,\label{eqn:1-from-thm:X-index-additive-on-triangles-with-error-term}\\
\newindx{\CX}(D) - \newindx{\CX}(Y) + \newindx{\CX}(B)&=0.\label{eqn:2-from-thm:X-index-additive-on-triangles-with-error-term}
\end{align}
As \(Z\in\add(\CX*\sus\CX)\), the Special Case gives 
\begin{equation}\label{eqn:3-from-thm:X-index-additive-on-triangles-with-error-term}
\newindx{\CX}(A) - \newindx{\CX}(Y) + \newindx{\CX}(Z) 
	= \tensor[]{\theta}{_{\CX}}\big([\Im \tensor*[]{H}{_{\CX}}(\eta)]\big).
\end{equation}
And the sum 
\(
\eqref{eqn:1-from-thm:X-index-additive-on-triangles-with-error-term} 
	- \eqref{eqn:2-from-thm:X-index-additive-on-triangles-with-error-term}
	+ \eqref{eqn:3-from-thm:X-index-additive-on-triangles-with-error-term} 
\) 
is
\(
\newindx{\CX}(A) - \newindx{\CX}(B) + \newindx{\CX}(C) 
	= \tensor[]{\theta}{_{\CX}}\big([\Im \tensor*[]{H}{_{\CX}}(\eta)]\big)
	\). 
Lastly, we observe that \(\tensor*[]{H}{_{\CX}}(\zeta)\) being epic implies that 
\(\Im \tensor*[]{H}{_{\CX}}(\gamma) \iso \Im \tensor*[]{H}{_{\CX}}(\eta)\) 
as \(\gamma\zeta = \eta\), so 
\(
\newindx{\CX}(A) - \newindx{\CX}(B) + \newindx{\CX}(C) 
	= \tensor[]{\theta}{_{\CX}}\big([\Im \tensor*[]{H}{_{\CX}}(\eta)]\big) 
	= \tensor[]{\theta}{_{\CX}}\big([\Im \tensor*[]{H}{_{\CX}}(\gamma)]\big)
\).
\end{proof}


\section{Understanding \texorpdfstring{\(\tensor[]{K}{_{0}}(\CC,\tensor[]{\BE}{_{\CX}},\tensor[]{\fs}{_{\CX}})\)}{K0CXEXsX}}
\label{sec:descriptions-of-K0CXEXsX}

In this section, we study the 
abelian group 
\(\tensor[]{K}{_{0}}(\CC,\tensor[]{\BE}{_{\CX}},\tensor[]{\fs}{_{\CX}})\). 
We work under the following setup, but specialise further in Subsection~\ref{subsec:relation-to-AR-angles} in which we prove Theorem~\ref{thmx:analogue-of-PPPP-Prop-4-11}.

\begin{setup}
\label{setup:2}
Let \(\field\) be a field. Suppose \(\CC\) is a skeletally small, \(\field\)-linear, \(\Hom\)-finite, idempotent complete, triangulated category. 
Suppose \(\CX\sse\CT\) are additive (see Subsection~\ref{subsec:conventions-notation}\ref{item:def-additive-category}), contravariantly finite subcategories, such that \(\CT\) (and hence also \(\CX\)) is rigid. 
Furthermore, we define an additive, rigid subcategory \(\CY\sse\CC\), which is closed under direct summands, by \(\ind\CT = \ind\CX \dot\cup \ind\CY\). 
\end{setup}

Note that this is a special case of Setup~\ref{setup:1}, and the assumptions imply \(\CC\) is Krull-Schmidt.

\subsection{The canonical surjection \texorpdfstring{\(Q\)}{Q} between relative Grothendieck groups}
\label{subsec:map-between-relative-grothendieck-groups}

In this subsection, we study the kernel of the canonical homomorphism 
\(Q \colon \tensor[]{K}{_{0}}(\CC,\tensor[]{\BE}{_{\CT}},\tensor[]{\fs}{_{\CT}}) \onto \tensor[]{K}{_{0}}(\CC,\tensor[]{\BE}{_{\CX}},\tensor[]{\fs}{_{\CX}})\)
induced by the containment \(\CX \sse \CT\) of rigid subcategories in a triangulated category \(\CC\). 
An abstract description of \(\Ker Q\) is given in \eqref{eqn:kernel-of-Q}, but we give another  in Proposition~\ref{prop:KerQ-general-description} 
using the homomorphism \(\tensor[]{\theta}{_{\CT}}\) from Proposition~\ref{prop:definition-of-theta-sub-CX}.
One can compute \(\Ker Q\) more explicitly in a special case, as we do in Theorem~\ref{thm:KerQ-T-locally-bounded-has-right-almost-split-maps-description}.

Recall that since \(\CC\) is triangulated, it has the structure of an extriangulated category, which we denote by \((\CC,\BE,\fs)\), where \(\BE(C,A)=\CC(C,\sus A)\) for objects \(A,C\in\CC\) (see Example~\ref{example:triangulated-category-is-extriangulated}). 
In addition, the subcategories \(\CX\) and \(\CT\) induce relative extriangulated structures on \(\CC\), namely there are extriangulated categories 
\((\CC,\tensor[]{\BE}{_{\CX}},\tensor[]{\fs}{_{\CX}})\) and \((\CC,\tensor[]{\BE}{_{\CT}},\tensor[]{\fs}{_{\CT}})\) (see Subsection~\ref{subsec:extriangulated-categories}). 
In the notation of Remark~\ref{rem:grothendieck-group-of-relative-structure}, we obtain 
the following diagram of subgroups in \(\tensor*[]{K}{_{0}^{\sp}}(\CC)\) using that \(0\sse\CX\sse\CT\sse\CC\). Note that \(\CI = \tensor[]{\CI}{_{0}}\).
\[
\begin{tikzcd}[row sep=0.4cm]
\tensor*[]{K}{_{0}^{\sp}}(\CC) \arrow[dash]{d}\\
\tensor[]{\CI}{_{0}}
	=  \left.\big\lan\,[A\tensor*[]{]}{_{\CC}^{\sp}}-[B\tensor*[]{]}{_{\CC}^{\sp}}+[C\tensor*[]{]}{_{\CC}^{\sp}} 
			\,\middle|\, [A \to B \to C] = \fs(\delta) \text{ for some } \delta\in\BE(C,A) \,\big\ran\right.
		\arrow[dash]{d}\\
\tensor[]{\CI}{_{\CX}} 
	= \left.\big\lan\, [A\tensor*[]{]}{_{\CC}^{\sp}}-[B\tensor*[]{]}{_{\CC}^{\sp}}+[C\tensor*[]{]}{_{\CC}^{\sp}} \,\middle|\ [A \to B \to C] = \fs(\delta) \text{ for some } \delta\in\tensor[]{\BE}{_{\CX}}(C,A) \,\big\ran\right.
	\arrow[dash]{d}\\
\tensor[]{\CI}{_{\CT}}	
	= \left.\big\lan\, [A\tensor*[]{]}{_{\CC}^{\sp}}-[B\tensor*[]{]}{_{\CC}^{\sp}}+[C\tensor*[]{]}{_{\CC}^{\sp}} \,\middle|\ [A \to B \to C] = \fs(\delta) \text{ for some } \delta\in\tensor[]{\BE}{_{\CT}}(C,A) \,\big\ran\right.
	\arrow[dash]{d}\\
0
\end{tikzcd}
\]
See also \cite[p.\ 16]{JorgensenShah-grothendieck-groups-of-d-exangulated-categories-and-a-modified-CC-map}.

In particular, there is a canonical surjection 
\begin{equation}
\label{eqn:homomorphism-Q}
\begin{tikzcd}[ampersand replacement=\&, column sep=2cm]
\begin{aligned}[t]
\tensor[]{K}{_{0}}(\CC,\tensor[]{\BE}{_{\CT}},\tensor[]{\fs}{_{\CT}}) \\
[C\tensor[]{]}{_{\CT}}
\end{aligned} 
	\arrow[two heads]{r}{Q \deff \tensor*[]{Q}{_{\CX}^{\CT}}}  \arrow[maps to, yshift={-0.75cm}]{r}\&
\begin{aligned}[t]
&\tensor[]{K}{_{0}}(\CC,\tensor[]{\BE}{_{\CX}},\tensor[]{\fs}{_{\CX}}) \\
&[C\tensor[]{]}{_{\CX}},
\end{aligned}
\end{tikzcd}
\end{equation}
where 
\begin{align}
\Ker Q 
	&= \left.\big\lan\, [A\tensor[]{]}{_{\CT}}-[B\tensor[]{]}{_{\CT}}+[C\tensor[]{]}{_{\CT}}  \,\middle|\, [A \to B \to C] = \fs(\delta) \text{ for some } \delta\in\tensor[]{\BE}{_{\CX}}(C,A) \,\big\ran\right.\nonumber\\
	&= \Braket{ [A\tensor[]{]}{_{\CT}}-[B\tensor[]{]}{_{\CT}}+[C\tensor[]{]}{_{\CT}} | 
\begin{aligned}
	&\text{there is a triangle } 
	\begin{tikzcd}[column sep=0.5cm, ampersand replacement=\&]
		A \arrow{r}\& B\arrow{r}\& C\arrow{r}{\delta} \&\sus A
	\end{tikzcd} \\[-7pt]
	&\text{in }\CC \text{ with } \tensor*[]{H}{_{\CX}}(\delta) = 0
\end{aligned}
}.\label{eqn:kernel-of-Q}
\end{align}

\begin{rem}
\label{remark19}

There is a commutative diagram
\[
\begin{tikzcd}
{}
&\tensor*[]{K}{_{0}^{\sp}}(\CC) 
	\arrow{dl}[swap]{\newindx{\CT}}
	\arrow{dr}{\newindx{\CX}}
&\\
\tensor[]{K}{_{0}}(\CC,\tensor[]{\BE}{_{\CT}},\tensor[]{\fs}{_{\CT}})
	\arrow[two heads]{rr}[swap]{Q = \tensor*[]{Q}{_{\CX}^{\CT}}}
&& \tensor[]{K}{_{0}}(\CC,\tensor[]{\BE}{_{\CX}},\tensor[]{\fs}{_{\CX}}),
\end{tikzcd}
\]
because 
\(Q\circ \newindx{\CT}(C) 
	= Q([C\tensor[]{]}{_{\CT}}) 
	= [C\tensor[]{]}{_{\CX}}
	= \newindx{\CX}(C)\).

\end{rem}

\begin{rem}
\label{remark20}

In the following proposition, which gives a description of \(\Ker Q\) in terms of \(\tensor[]{\theta}{_{\CT}}\), note that \(L\in\rmod{\CT}\) is a functor \(\tensor[]{\CT}{^{\op}}\to \Ab\) and that \(\CX\sse\CT\), whence the restriction \(\restr{L}{\CX}\) makes sense. 
Note also that 
\[
\restr{\tensor*[]{H}{_{\CT}}(C)}{\CX} 
	= \restr{\left(\restr{\CC(-,C)}{\CT} \right)}{\CX}
	= \restr{\CC(-,C)}{\CX}
	= \tensor*[]{H}{_{\CX}}(C).
\]

\end{rem}

\begin{prop}
\label{prop:KerQ-general-description}

We have 
\(\Ker Q = \left.\big\lan\, \tensor[]{\theta}{_{\CT}}([L]) \,\middle|\, L\in\rmod{\CT} \text{ \emph{with} } \restr{L}{\CX} = 0 \,\big\ran\right.\). 

\end{prop}

\begin{proof}

\((\spse)\)\;\; 
Let \(L\in\rmod{\CT}\) satisfy \(\restr{L}{\CX} = 0\). 
By Section~\ref{subsec:approximation}\ref{item:3iii}, there is a triangle 
\begin{equation}\label{eqn:triangle-from-prop:KerQ-general-description}
\begin{tikzcd}
	\tensor*[]{\sus}{^{-1}}A \arrow{r}& \tensor[]{T}{_{1}} \arrow{r}& \tensor[]{T}{_{0}} \arrow{r}{\tensor*[]{\tau}{_{0}}}& A
\end{tikzcd}
\end{equation}
in \(\CC\) with \(\tensor[]{T}{_{i}}\in\CT\) such that \(L \iso \tensor*[]{H}{_{\CT}}(A)\). 
Note that \(\tensor*[]{H}{_{\CT}}(\tensor*[]{\tau}{_{0}})\) is epic as \(\tensor*[]{H}{_{\CT}}(\sus \tensor[]{T}{_{1}}) = 0\), so 
\begin{equation}\label{eqn:Image-of-H-tau-is-L}
\Im \tensor*[]{H}{_{\CT}}(\tensor*[]{\tau}{_{0}}) \iso \tensor*[]{H}{_{\CT}}(A) \iso L.
\end{equation}

We have \(\tensor*[]{H}{_{\CX}}(A) = \restr{\tensor*[]{H}{_{\CT}}(A)}{\CX} \iso \restr{L}{\CX} = 0\) by Remark~\ref{remark20} and our assumption on \(L\). 
In particular, \(\tensor*[]{H}{_{\CX}}(\tensor*[]{\tau}{_{0}}) = 0\), so \eqref{eqn:triangle-from-prop:KerQ-general-description} is one of the triangles in \eqref{eqn:kernel-of-Q}. 
Hence,
\begin{align*}
\tensor[]{\theta}{_{\CT}}([L]) 
	&=\tensor[]{\theta}{_{\CT}}\big([\Im \tensor*[]{H}{_{\CT}}(\tensor*[]{\tau}{_{0}})]\big)
		&& \text{ using } \eqref{eqn:Image-of-H-tau-is-L}\\
	&= [\tensor*[]{\sus}{^{-1}}A\tensor[]{]}{_{\CT}} - [\tensor[]{T}{_{1}}\tensor[]{]}{_{\CT}} + [\tensor[]{T}{_{0}}\tensor[]{]}{_{\CT}} 
		&& \text{ by Theorem}~ \ref{thm:X-index-additive-on-triangles-with-error-term}.
\end{align*}
Hence, \(\tensor[]{\theta}{_{\CT}}([L])\in\Ker Q\). 

\((\sse)\)\;\; 
Conversely, consider a generator of \(\Ker Q\), that is, an element 
\([A\tensor[]{]}{_{\CT}} - [B\tensor[]{]}{_{\CT}} + [C\tensor[]{]}{_{\CT}}\) 
of \(\tensor[]{K}{_{0}}(\CC,\tensor[]{\BE}{_{\CT}},\tensor[]{\fs}{_{\CT}})\) for which there is a triangle 
\(\begin{tikzcd}[column sep=0.5cm]
		A \arrow{r}& B\arrow{r}& C\arrow{r}{\gamma} &\sus A
\end{tikzcd}\) 
in \(\CC\) such that \(\tensor*[]{H}{_{\CX}}(\gamma) = 0\); see \eqref{eqn:kernel-of-Q}. 
Consider \(\Im \tensor*[]{H}{_{\CT}}(\gamma)\in\rmod{\CT}\). 
Note \(\tensor[]{\theta}{_{\CT}}\big([\Im \tensor*[]{H}{_{\CT}}(\gamma)]\big) = [A\tensor[]{]}{_{\CT}} - [B\tensor[]{]}{_{\CT}} + [C\tensor[]{]}{_{\CT}}\) by Theorem~\ref{thm:X-index-additive-on-triangles-with-error-term}. 
Thus, it suffices to show that \(\restr{\Im \tensor*[]{H}{_{\CT}}(\gamma)}{\CX}\) is trivial. 
Observe that 
\begin{align*}
\restr{(\Im \tensor*[]{H}{_{\CT}}(\gamma))}{\CX}&=\Im ( \restr{(\tensor*[]{H}{_{\CT}}(\gamma))}{\CX} )&& 
\text{ as } \restr{(-)}{\CX}\colon \rmod{\CT} \to \rmod{\CX} \text{ is an exact functor}\\
			&= \Im \tensor*[]{H}{_{\CX}}(\gamma) && \text{ by Remark}~\ref{remark20}\\
			&=0.&&\qedhere
\end{align*}
\end{proof}

Recall that a \(\field\)-linear category \(\CA\) is called \emph{locally bounded} if \(\CA\) is \(\Hom\)-finite and, for each indecomposable object \(A\in\CA\), there are only finitely many objects 
\(B\in\ind\CA\) such that \(\CA(A,B)\neq 0\) or \(\CA(B,A)\neq 0\); see \cite[p.\ 160]{LenzingReiten-hereditary-noetherian-categories-of-positive-euler-characteristic}. 
The category \(\CA\) \emph{has right almost split morphisms} if each indecomposable object \(A\in\CA\) admits a right almost split morphism \(C \to A\) in the sense of \cite[p.\ 454]{AuslanderReiten-Rep-theory-of-Artin-algebras-IV}; see \cite[p.\ 285]{Auslander-Rep-theory-of-Artin-algebras-II}, where right almost split morphisms were called `almost splittable'.

\begin{thm}
\label{thm:KerQ-T-locally-bounded-has-right-almost-split-maps-description}

If \(\CT\) is locally bounded and has right almost split morphisms, then
\[
\Ker Q = \left.\big\lan\, \tensor[]{\theta}{_{\CT}}([\tensor[]{S}{_{Y}}]) \,\middle|\,  Y\in\ind\CY \,\big\ran\right.,
\]
where \(\tensor[]{S}{_{Y}}\in\rmod{\CT}\) is the simple top of \(\CT(-,Y)\). 

\end{thm}

\begin{proof}

Local boundedness and \(\Hom\)-finiteness imply each \(L\in\rmod{\CX}\) has finite length in \(\rMod{\CX}\) (the category of all additive functors \(\tensor[]{\CX}{^{\op}} \to \Ab\)), hence a finite composition series. 
The simple quotients in the series are simple tops \(\tensor[]{S}{_{T}}\) of objects \(\CT(-,T)\) for \(T\in\ind \CT\); see \cite[Prop.\ 2.3(b)]{Auslander-Rep-theory-of-Artin-algebras-II}. 
Since \(\CT\) has right almost split morphisms, the \(\tensor[]{S}{_{T}}\) are in \(\rmod{\CT}\) by \cite[Cor.\ 2.6]{Auslander-Rep-theory-of-Artin-algebras-II}. 
Finally, note that \(\restr{L}{\CX}=0\) if and only if each composition factor \(\tensor[]{S}{_{T}}\) satisfies \(\restr{\tensor[]{S}{_{T}}}{\CX} = 0\), that is, if and only if \(T\in\ind\CY\).
Hence, 
\begin{align*}
\left.\big\lan\, \tensor[]{\theta}{_{\CT}}([\tensor[]{S}{_{Y}}]) \,\middle|\,  Y\in\ind\CY \,\big\ran\right.
	&= \left.\big\lan\, \tensor[]{\theta}{_{\CT}}([L]) \,\middle|\,  L\in\rmod{\CT} \text{ with } \restr{L}{\CX} = 0 \,\big\ran\right.&&\\
	&= \Ker Q,&& 
\end{align*}
where the last equality follows from \cref{prop:KerQ-general-description}. 
\end{proof}


\subsection{Connection to Auslander-Reiten angles and modified Caldero-Chapoton maps}
\label{subsec:relation-to-AR-angles}

We specialise further in this subsection in order to produce in Theorem~\ref{thm:Ker-Q-in-terms-of-AR-n-plus-2-angles} a presentation of \(\Ker Q\) in terms of Auslander-Reiten angles. In particular, we assume \(\CT\) is a locally bounded, \(n\)-cluster tilting subcategory; see the setup below.

\begin{setup}
\label{setup:4}

Let \(\field\) be a field and let \(n\geq 2\) be an integer. Suppose \(\CC\) is a skeletally small, \(\field\)-linear, \(\Hom\)-finite, idempotent complete, triangulated category with a Serre functor \(\BS\). 
Suppose \(\CX\sse\CT\) are additive (see Subsection~\ref{subsec:conventions-notation}\ref{item:def-additive-category}), contravariantly finite subcategories, 
In addition, suppose \(\CT\) is a locally bounded, \(n\)-cluster tilting subcategory of \(\CC\) (see Definition~\ref{def:n-cluster-tilting} below). 
We still define an additive rigid subcategory \(\CY\sse\CC\), which is closed under direct summands, by \(\ind\CT = \ind\CX \dot\cup \ind\CY\). 

\end{setup}

This is a special case of Setup~\ref{setup:2}. 
Let us recall the definition of \(n\)-cluster tilting.

\begin{defn}
\label{def:n-cluster-tilting}
\cite{Iyama-maximal-orthogonal-subcategories-of-triangulated-categories-satisfying-serre-duality}, 
\cite[Sec.\ 3]{IyamaYoshino-mutation-in-tri-cats-rigid-CM-mods} 
Suppose \(\CC\) is an idempotent complete, triangulated category with an additive subcategory \(\CT\). 
For an integer \(n\geq 2\), we call \(\CT\) an \emph{\(n\)-cluster tilting subcategory of \(\CC\)} if: 
\begin{enumerate}[label = (\roman*)]
	\item \(\CT\) is functorially finite; and 	
	\item 
		\(\begin{aligned}[t]
			\CT 	&= \Set{X\in \CC | \tensor*[]{\Ext}{_{\CC}^{i}}(\CT,X) = 0 \text{ for all } 1 \leq i \leq n-1} \\
		 			&= \Set{X\in \CC | \tensor*[]{\Ext}{_{\CC}^{i}}(X,\CT) = 0 \text{ for all } 1 \leq i \leq n-1}.
		\end{aligned}\)
\end{enumerate}
\end{defn}

Our first aim in this subsection is to prove Theorem~\ref{thmx:analogue-of-PPPP-Prop-4-11} from Section~\ref{sec:introduction}. For this, we recall the following preliminary definition and then 
the triangulated index of \cite{Jorgensen-tropical-friezes-and-the-index-in-higher-homological-algebra}.

\begin{defn}
\label{def:tower-of-triangles}

\cite[Def.\ 3.1]{Jorgensen-tropical-friezes-and-the-index-in-higher-homological-algebra} 
By a \emph{tower of triangles in \(\CC\)}, we mean a diagram 
\[
\begin{tikzcd}[column sep=0.4cm]
{}
&\tensor*[]{A}{_{m-1}}
	\arrow{dr}{}
	\arrow{rr}
&&\tensor*[]{A}{_{m-2}}
	\arrow[end anchor={[xshift=-8pt, yshift=3pt]}]{dr}{}
	\arrow{rr}
&&\cdots
	\arrow{rr}
&&\tensor*[]{A}{_{2}}
	\arrow{dr}
	\arrow{rr}
&&\tensor*[]{A}{_{1}}
	\arrow{dr}{}
& \\
\tensor*[]{A}{_{m}}
	\arrow{ur}{}
&&\tensor*[]{V}{_{m-1.5}}
	\arrow{ur}{}
	\arrow[rightsquigarrow]{ll}{}
&&\tensor*[]{V}{_{m-2.5}}
	\arrow[rightsquigarrow]{ll}{}
&\cdots
&\tensor*[]{V}{_{2.5}}
	\arrow{ur}{}
&&\tensor*[]{V}{_{1.5}}
	\arrow{ur}{}
	\arrow[rightsquigarrow]{ll}{}
&&\tensor*[]{A}{_{0}}, 
	\arrow[rightsquigarrow]{ll}{}
\end{tikzcd}
\]
in \(\CC\), where  
\(m\geq 2\) is an integer, 
an arrow 
	\(\begin{tikzcd}[column sep=0.6cm]
	B \arrow[rightsquigarrow]{r}& C
	\end{tikzcd}\)
represents a morphism
	\(\begin{tikzcd}[column sep=0.6cm]
	B \arrow{r}& \sus C,
	\end{tikzcd}\) 
each oriented triangle is a triangle in \(\CC\), 
and each non-oriented triangle commutes. 

\end{defn}

\begin{defn}
\label{def:triangulated-index-wrt-n-cluster-tilting-subcategory}

\cite[Lem.\ 3.2, Def.\ 3.3]{Jorgensen-tropical-friezes-and-the-index-in-higher-homological-algebra} 
Let \(C\in\CC\) be arbitrary. There is a tower of triangles 
\begin{equation}\label{eqn:tower-of-triangles-of-T-covers-for-C}
\begin{tikzcd}[column sep=0.5cm]
	&\tensor[]{T}{_{n-2}}\arrow{dr}{\tensor*[]{\tau}{_{n-2}}}\arrow{rr}&&\tensor[]{T}{_{n-3}}\arrow[end anchor={[xshift=-8pt, yshift=3pt]}]{dr}{\tensor*[]{\tau}{_{n-3}}}\arrow{r}&\cdots\arrow{r}&\tensor[]{T}{_{1}}\arrow{dr}{\tensor*[]{\tau}{_{1}}}\arrow{rr}&&\tensor[]{T}{_{0}}\arrow{dr}{\tensor*[]{\tau}{_{0}}}& \\
	\tensor[]{T}{_{n-1}}\arrow{ur}{}&&V_{n-2.5}\arrow{ur}{}\arrow[rightsquigarrow]{ll}{\tensor*[]{\nu}{_{n-2.5}}}&&\hspace{0.5cm}\cdots\hspace{0.5cm}\arrow[start anchor={[xshift=8pt, yshift=3pt]}]{ur}{}\arrow[rightsquigarrow]{ll}{\tensor*[]{\nu}{_{n-3.5}}}&&V_{0.5}\arrow{ur}{}\arrow[rightsquigarrow]{ll}{\tensor*[]{\nu}{_{0.5}}}&&C, \arrow[rightsquigarrow]{ll}{\tensor*[]{\nu}{_{-0.5}}}
\end{tikzcd}
\end{equation}
where each \(\tensor[]{\tau}{_{i}}\) is a minimal right \(\CT\)-approximation. 
The \emph{triangulated index of \(C\) with respect to \(\CT\)} is 
\[
\indxx{\CT}(C) \deff \sum_{i=0}^{n-1}(-1\tensor[]{)}{^{i}}[\tensor[]{T}{_{i}}\tensor*[]{]}{_{\CT}^{\sp}},
\]
viewed as an element of the split Grothendieck group \(\tensor*[]{K}{_{0}^{\sp}}(\CT)\). 
This assignment induces a surjective homomorphism 
\(\indxx{\CT} \colon \tensor*[]{K}{_{0}^{\sp}}(\CC) \onto \tensor*[]{K}{_{0}^{\sp}}(\CT)\); see \cite[Rem.\ 2.14]{Fedele-grothendieck-groups-of-triangulated-categories-via-cluster-tilting-subcategories}.
\end{defn}

We can now prove Theorem~\ref{thmx:analogue-of-PPPP-Prop-4-11} from Section~\ref{sec:introduction}, which can be viewed as a higher cluster tilting analogue of \cite[Prop.\ 4.11]{PadrolPaluPilaudPlamondon-associahedra-for-finite-type-cluster-algebras-and-minimal-relations-between-g-vectors}. Note that this result requires neither the existence of a Serre functor nor that \(\CT\) is locally bounded.

\begin{thm}
\label{thm:analogue-of-PPPP-Prop-4-11}
There is an isomorphism
\[
\begin{tikzcd}[ampersand replacement=\&]
	\begin{aligned}[t]
		\tensor[]{K}{_{0}}(\CC,\tensor[]{\BE}{_{\CT}},\tensor[]{\fs}{_{\CT}}) &\\
		\newindx{\CT}(C) = [C\tensor[]{]}{_{\CT}}&
	\end{aligned}
	\arrow{r}{\tensor[]{\Phi}{_{\CT}}}	\arrow[maps to, yshift={-0.75cm}]{r}\&
	\begin{aligned}[t]
		&\tensor*[]{K}{_{0}^{\sp}}(\CT) \\
		&\indxx{\CT}(C) 
	\end{aligned}
\end{tikzcd}
\]
of abelian groups, making the following diagram commute. 
\[
\begin{tikzcd}[row sep=1cm, column sep=2cm]
\tensor*[]{K}{_{0}^{\sp}}(\CC) 
	\arrow{d}[swap]{\newindx{\CT}(-) = [-\tensor[]{]}{_{\CT}}} 
	\arrow{dr}{\indxx{\CT}(-)}& \\
\tensor[]{K}{_{0}}(\CC,\tensor[]{\BE}{_{\CT}},\tensor[]{\fs}{_{\CT}}) \arrow{r}{\iso}[swap]{\tensor[]{\Phi}{_{\CT}}}  & \tensor*[]{K}{_{0}^{\sp}}(\CT)
\end{tikzcd}
\]

\end{thm}

\begin{proof}
The assignment
\(\tensor[]{\Phi}{_{\CT}}\) gives a well-defined homomorphism, because it sends relations in 
\(
\tensor[]{K}{_{0}}(\CC,\tensor[]{\BE}{_{\CT}},\tensor[]{\fs}{_{\CT}}) = \tensor*[]{K}{_{0}^{\sp}}(\CC)/\tensor[]{\CI}{_{\CT}}
\)
to 0. 
Namely, a relation arises from a triangle 
\(\begin{tikzcd}[column sep=0.5cm]
	\hspace{-5pt}A \arrow{r}& B \arrow{r}& C \arrow{r}{\gamma}& \sus A
\end{tikzcd}\)
with \(\tensor*[]{H}{_{\CT}}(\gamma) = 0\), which satisfies 
\(
\hspace{3pt}
\indxx{\CT}(A) - \indxx{\CT}(B) + \indxx{\CT}(C) = 0
\) 
by \cite[Lem.\ 3.7]{Jorgensen-tropical-friezes-and-the-index-in-higher-homological-algebra}.

We claim that 
\[
\begin{tikzcd}[ampersand replacement=\&]
	\begin{aligned}[t]
		\tensor*[]{K}{_{0}^{\sp}}(\CT)& 
		\\
		[T\tensor*[]{]}{_{\CT}^{\sp}}& 
	\end{aligned}
	\arrow{r}{\tensor[]{\Psi}{_{\CT}}}	
	\arrow[maps to, yshift={-0.75cm}]{r}
	\&
	\begin{aligned}[t]
		&\tensor[]{K}{_{0}}(\CC,\tensor[]{\BE}{_{\CT}},\tensor[]{\fs}{_{\CT}})
		 \\
		&[T\tensor[]{]}{_{\CT}}.
	\end{aligned}
\end{tikzcd}
\]
defines an inverse homomorphism to \(\tensor[]{\Phi}{_{\CT}}\). 
This formula clearly defines a homomorphism of groups.

To simplify notation, put 
\(\Phi\deff \tensor[]{\Phi}{_{\CT}}\)
and
\(\Psi\deff \tensor[]{\Psi}{_{\CT}}\). 
First, note that \(\Phi \circ \Psi = \iden{\tensor*[]{K}{_{0}^{\sp}}(\CT)}\). 
Conversely, let \(C\in\CC\) be arbitrary. Then there is a tower of triangles \eqref{eqn:tower-of-triangles-of-T-covers-for-C},
and we have
\begin{align*}
\Psi\Phi([C\tensor[]{]}{_{\CT}}) 
			&= \Psi(\indxx{\CT}(C))&& \text{ by definition of } \Phi\\
			&= \Psi\left( \sum_{i=0}^{n-1}(-1\tensor[]{)}{^{i}}[\tensor[]{T}{_{i}}\tensor*[]{]}{_{\CT}^{\sp}} \right) && \text{ by definition of } \indxx{\CT}(C)\\
			&=  \sum_{i=0}^{n-1}(-1\tensor[]{)}{^{i}}[\tensor[]{T}{_{i}}\tensor[]{]}{_{\CT}} && \text{ by definition of } \Psi.\\
\end{align*}
Thus, it suffices to show 
\([C\tensor[]{]}{_{\CT}} = \sum_{i=0}^{n-1}(-1\tensor[]{)}{^{i}}[\tensor[]{T}{_{i}}\tensor[]{]}{_{\CT}}\). 
Note that, for \(0 \leq i \leq n-2\), each \(\tensor[]{\tau}{_{i}}\) in \eqref{eqn:tower-of-triangles-of-T-covers-for-C} is a minimal right \(\CT\)-approximation,  so \(\tensor*[]{H}{_{\CT}}(\tensor[]{\tau}{_{i}})\) is surjective. 
Hence, \(\tensor*[]{H}{_{\CT}}(\tensor[]{\nu}{_{i-0.5}}) = 0\) for each \(0 \leq i \leq n-2\). 
Set 
\(\tensor[]{V}{_{-0.5}}\deff C\) 
and 
\(\tensor[]{V}{_{n-1.5}}\deff \tensor[]{T}{_{n-1}}\). 
Then applying Theorem~\ref{thm:X-index-additive-on-triangles-with-error-term} yields 
\(
[\tensor[]{V}{_{i+0.5}}\tensor[]{]}{_{\CT}} - [\tensor[]{T}{_{i}}\tensor[]{]}{_{\CT}} + [\tensor[]{V}{_{i-0.5}}\tensor[]{]}{_{\CT}} = 0 
\)
for each \(0 \leq i \leq n-2\). 
Taking their alternating sum yields 
\([C\tensor[]{]}{_{\CT}} - \sum_{i=0}^{n-1}(-1\tensor[]{)}{^{i}}[\tensor[]{T}{_{i}}\tensor[]{]}{_{\CT}} = 0\), 
and hence we see that 
\(\Psi\Phi([C\tensor[]{]}{_{\CT}}) = [C\tensor[]{]}{_{\CT}}\), i.e.\ \(\Psi\Phi\) is the identity on \(\tensor[]{K}{_{0}}(\CC,\tensor[]{\BE}{_{\CT}},\tensor[]{\fs}{_{\CT}})\).
\end{proof}

Our last result of this subsection gives an interpretation of \(\Ker Q\) in terms of Auslander-Reiten \((n+2)\)-angles in \(\CT\). 
Let us recall what this means. 
An \emph{Auslander-Reiten \((n+2)\)-angle} 
\(
\begin{tikzcd}[column sep=0.5cm]
	\tensor[]{T}{_{n+1}} \arrow{r} & \tensor[]{T}{_{n}} \arrow{r} & \cdots \arrow{r} & \tensor[]{T}{_{1}} \arrow{r} & \tensor[]{T}{_{0}}
\end{tikzcd}
\)
in \(\CT\) amounts to a tower of triangles
\begin{equation}
\label{eqn:tower-of-triangles-for-AR-angle}
\begin{tikzcd}[column sep=0.7cm]
	&\tensor[]{T}{_{n}}\arrow{dr}{\tensor[]{\beta}{_{n}}}\arrow{rr}
	&&\tensor[]{T}{_{n-1}}\arrow[end anchor={[xshift=-8pt, yshift=3pt]}]{dr}{\tensor[]{\beta}{_{n-1}}}\arrow{r}
	&\cdots\arrow{r}
	&\tensor[]{T}{_{2}}\arrow{dr}{\tensor[]{\beta}{_{2}}}\arrow{rr}
	&&\tensor[]{T}{_{1}}\arrow{dr}{\tensor[]{\beta}{_{1}}}
	& \\
	\tensor[]{T}{_{n+1}}\arrow{ur}{\tensor[]{\beta}{_{n+1}}}
	&&\tensor*[]{X}{_{n-0.5}}\arrow{ur}{\tensor[]{\alpha}{_{n-0.5}}}\arrow[rightsquigarrow]{ll}{\tensor[]{\gamma}{_{n-0.5}}}
	&&\hspace{0.5cm}\cdots\hspace{0.5cm}\arrow[start anchor={[xshift=8pt, yshift=3pt]}]{ur}{\tensor[]{\alpha}{_{2.5}}}\arrow[rightsquigarrow]{ll}{\tensor[]{\gamma}{_{n-1.5}}}
	&&\tensor*[]{X}{_{1.5}}\arrow{ur}{\tensor[]{\alpha}{_{1.5}}}\arrow[rightsquigarrow]{ll}{\tensor[]{\gamma}{_{1.5}}}
	&&\tensor[]{T}{_{0}} \arrow[rightsquigarrow]{ll}{\tensor[]{\gamma}{_{0.5}}}
\end{tikzcd}
\end{equation}
in \(\CC\), where:
\begin{itemize}

	\item \(\tensor[]{\beta}{_{1}}\) is minimal right almost split in \(\CT\);
	
	\item \(\tensor[]{\beta}{_{n+1}}\) is minimal left almost split in \(\CT\);
	
	\item \(\tensor[]{\beta}{_{2}},\ldots,\tensor[]{\beta}{_{n}}\) are minimal right \(\CT\)-approximations; and 
	
	\item \(\tensor[]{\alpha}{_{1.5}},\ldots,\tensor[]{\alpha}{_{n-0.5}}\) are minimal left \(\CT\)-approximations;

\end{itemize}
see \cite[Def.\ 3.8]{IyamaYoshino-mutation-in-tri-cats-rigid-CM-mods}. 
By \cite[Thm.\ 3.10]{IyamaYoshino-mutation-in-tri-cats-rigid-CM-mods}, each \(\tensor[]{T}{_{0}}\in\ind\CT\) permits such an Auslander-Reiten \((n+2)\)-angle in \(\CT\) with \(\tensor[]{T}{_{n+1}} =  \BS \tensor*[]{\sus}{^{-n}}\tensor[]{T}{_{0}}\).

\begin{thm}
\label{thm:Ker-Q-in-terms-of-AR-n-plus-2-angles}

The canonical surjection \(Q = \tensor*[]{Q}{_{\CX}^{\CT}}\colon \tensor[]{K}{_{0}}(\CC,\tensor[]{\BE}{_{\CT}},\tensor[]{\fs}{_{\CT}})\to \tensor[]{K}{_{0}}(\CC,\tensor[]{\BE}{_{\CX}},\tensor[]{\fs}{_{\CX}})\) from \eqref{eqn:homomorphism-Q} satisfies
\[
\Ker Q 
	= \Braket{ \sum_{i=0}^{n+1}(-1\tensor[]{)}{^{i}}[\tensor[]{T}{_{i}}\tensor[]{]}{_{\CT}} | 
		\begin{aligned}
		&\tensor[]{T}{_{0}} = Y\in\ind\CY \text{ \emph{and} } 
			\begin{tikzcd}[column sep=0.5cm,ampersand replacement=\&]
			\tensor[]{T}{_{n+1}} \arrow{r}\& \tensor[]{T}{_{n}} \arrow{r}\& \cdots \arrow{r}\& \tensor[]{T}{_{1}} \arrow{r}\& \tensor[]{T}{_{0}}
			\end{tikzcd}\\[-7pt]
		&\text{\emph{is an Auslander-Reiten} } (n+2)\text{\emph{-angle in} }\CT
		\end{aligned}
		}.
\]

\end{thm}

\begin{proof}

The subcategory \(\CT\) has right almost split morphisms because Auslander-Reiten \((n+2)\)-angles exist, or by \cite[Props.\ 2.10, 2.11]{IyamaYoshino-mutation-in-tri-cats-rigid-CM-mods}. 
In addition, \(\CT\) is locally bounded by assumption (see Setup~\ref{setup:4}). Hence, Theorem~\ref{thm:KerQ-T-locally-bounded-has-right-almost-split-maps-description} applies. 

Thus, it suffices to show  
\begin{equation}
\label{eqn:theta-of-simple-formula}
 \sum_{i=0}^{n+1}(-1\tensor[]{)}{^{i}}[\tensor[]{T}{_{i}}\tensor[]{]}{_{\CT}} 
 	= \tensor[]{\theta}{_{\CT}}([\tensor[]{S}{_{\tensor[]{T}{_{0}}}}]).
\end{equation}
for \(\tensor[]{T}{_{0}}\in\ind\CT\), where 
\(\tensor[]{S}{_{\tensor[]{T}{_{0}}}}\in\rmod{\CT}\) is the simple top of \(\CT(-,\tensor[]{T}{_{0}})\) and 
\[
\begin{tikzcd}[ampersand replacement=\&]
	\tensor[]{T}{_{n+1}} \arrow{r}\& \tensor[]{T}{_{n}} \arrow{r}\& \cdots \arrow{r}\& \tensor[]{T}{_{1}} \arrow{r}\& \tensor[]{T}{_{0}}
\end{tikzcd}
\] 
is an Auslander-Reiten \((n+2)\)-angle with associated tower of triangles 
\eqref{eqn:tower-of-triangles-for-AR-angle}. 
The triangle 
\[
\begin{tikzcd}[column sep=1cm]
	\tensor*[]{X}{_{1.5}} \arrow{r}{\tensor[]{\alpha}{_{1.5}}}& \tensor[]{T}{_{1}}\arrow{r}{\tensor[]{\beta}{_{1}}}& \tensor[]{T}{_{0}} \arrow{r}{\tensor[]{\gamma}{_{0.5}}}& \sus \tensor*[]{X}{_{1.5}}
\end{tikzcd} 
\]
satisfies  
\(\Im \tensor*[]{H}{_{\CT}}(\tensor[]{\gamma}{_{0.5}}) 
	= \Cok \tensor*[]{H}{_{\CT}}(\tensor[]{\beta}{_{1}}) 
	= \tensor[]{S}{_{\tensor[]{T}{_{0}}}}\),
because \(\tensor[]{\beta}{_{1}}\) is right almost split in \(\CT\); 
see \cite[Cor.\ 2.6]{Auslander-Rep-theory-of-Artin-algebras-II}. 
Thus, by Theorem~\ref{thm:X-index-additive-on-triangles-with-error-term}, we have that 
\begin{equation}\label{eqn:star-from-proof-of-thm:Ker-Q-in-terms-of-AR-n-plus-2-angles}
	\newindx{\CT}(\tensor*[]{X}{_{1.5}}) - \newindx{\CT}(\tensor[]{T}{_{1}}) + \newindx{\CT}(\tensor[]{T}{_{0}})
	= \tensor[]{\theta}{_{\CT}}([\tensor[]{S}{_{\tensor[]{T}{_{0}}}}]).
\end{equation}
Set \(\tensor*[]{X}{_{n+0.5}} \deff \tensor[]{T}{_{n+1}}\). 
For \(2\leq i \leq n\), the triangle 
\[
\begin{tikzcd}[column sep=1.3cm]
	\tensor[]{X}{_{i+0.5}} \arrow{r}{\tensor[]{\alpha}{_{i+0.5}}}& \tensor[]{T}{_{i}}\arrow{r}{\tensor[]{\beta}{_{i}}}& \tensor[]{X}{_{i-0.5}}\arrow{r}{\tensor[]{\gamma}{_{i-0.5}}}& \sus \tensor[]{X}{_{i+0.5}}
\end{tikzcd}
\]
satisfies \(\Im \tensor*[]{H}{_{\CT}}(\tensor[]{\gamma}{_{i-0.5}}) = 0\), because \(\tensor[]{\beta}{_{i}}\) is a right \(\CT\)-approximation. 
So, by Theorem~\ref{thm:X-index-additive-on-triangles-with-error-term}, we also have that 
\stepcounter{equation}
\begin{equation}
\label{eqn:starstar-from-proof-of-thm:Ker-Q-in-terms-of-AR-n-plus-2-angles}
\tag*{\((\theequation\tensor[]{)}{_{i}}\)}
\newindx{\CT}(\tensor[]{X}{_{i+0.5}}) - \newindx{\CT}(\tensor[]{T}{_{i}}) + \newindx{\CT}(\tensor[]{X}{_{i-0.5}}) 
	= 0
\end{equation}
for each \(2\leq i \leq n\). 
Hence, the sum 
\[
\mbox{\eqref{eqn:star-from-proof-of-thm:Ker-Q-in-terms-of-AR-n-plus-2-angles}}
	+ \displaystyle\sum_{i=2}^{n}(-1\tensor[]{)}{^{i-1}} 	
\mbox{\ref{eqn:starstar-from-proof-of-thm:Ker-Q-in-terms-of-AR-n-plus-2-angles}}
\]
is precisely 
\eqref{eqn:theta-of-simple-formula}, which was to be shown. 
\end{proof}

\begin{rem}
\label{rem:K0CX-is-quotient-of-split-G-group-of-T}

It follows from Theorems~\ref{thm:analogue-of-PPPP-Prop-4-11} and \ref{thm:Ker-Q-in-terms-of-AR-n-plus-2-angles} that the group 
\(
\tensor[]{K}{_{0}}(\CC,\tensor[]{\BE}{_{\CX}},\tensor[]{\fs}{_{\CX}}) 
	\iso \tensor[]{K}{_{0}}(\CC,\tensor[]{\BE}{_{\CT}},\tensor[]{\fs}{_{\CT}}) / \Ker Q 
\)
is isomorphic to 
\[
\left.\tensor*[]{K}{_{0}^{\sp}}(\CT) \middle/ 
		\Braket{ \sum_{i=0}^{n+1}(-1\tensor[]{)}{^{i}}[\tensor[]{T}{_{i}}\tensor*[]{]}{_{\CT}^{\sp}} | 
		\begin{aligned}
		&\tensor[]{T}{_{0}}\in\ind\CY \text{ \emph{and} } 
			\begin{tikzcd}[column sep=0.5cm,ampersand replacement=\&]
			\tensor[]{T}{_{n+1}} \arrow{r}\& \tensor[]{T}{_{n}} \arrow{r}\& \cdots \arrow{r}\& \tensor[]{T}{_{1}} \arrow{r}\& \tensor[]{T}{_{0}}
			\end{tikzcd}\\[-7pt]
		&\text{\emph{is an Auslander-Reiten} } (n+2)\text{\emph{-angle in} }\CT
		\end{aligned}
		}.\right.
\]
\end{rem}

We finish this section with a remark on how our considerations here relate to the modified Caldero-Chapoton map of Holm--J{\o{}}rgensen \cite{HolmJorgensen-generalized-friezes-and-a-modified-caldero-chapoton-map-depending-on-a-rigid-object-1}, \cite{HolmJorgensen-generalized-friezes-and-a-modified-caldero-chapoton-map-depending-on-a-rigid-object-2} and the \(\CX\)-Caldero-Chapoton map from \cite{JorgensenShah-grothendieck-groups-of-d-exangulated-categories-and-a-modified-CC-map}.

\begin{rem}
\label{remark26}

In addition to assuming the conditions of Setup~\ref{setup:4}, in this remark we also suppose the following:
\begin{enumerate}[(\roman*)]
	\item \(\field\) is algebraically closed;
	\item \(n=2\);
	\item \(\CC\) is 2-Calabi-Yau;
	\item \(\CX\) is functorially finite; and
	\item \(\CT\) is part of a cluster structure in the sense of \cite[p.\ 1039]{BuanIyamaReitenScott-cluster-structures-for-2-calabi-yau-categories-and-unipotent-groups}. 
\end{enumerate}
Thus, for \(Y\in\ind\CY\), there is a pair of exchange triangles 
\[
\begin{tikzcd}
	Y \arrow{r}& A' \arrow{r}& \tensor*[]{Y}{^{*}} \arrow{r}& \sus Y 
\end{tikzcd}
\hspace{0.5cm}\text{and}\hspace{0.5cm}
\begin{tikzcd}
	\tensor*[]{Y}{^{*}}\arrow{r} & A \arrow{r}& Y \arrow{r}& \sus \tensor*[]{Y}{^{*}}
\end{tikzcd}
\]
that can be glued to obtain an Auslander-Reiten 4-angle 
\[
\begin{tikzcd}[column sep=0.7cm]
&A' \arrow{dr}\arrow{rr}&& A\arrow{dr} & \\
Y\arrow{ur} &&\tensor*[]{Y}{^{*}}\arrow[rightsquigarrow]{ll}\arrow{ur}&&Y\arrow[rightsquigarrow]{ll}
\end{tikzcd}
\]
in \(\CT\). 
It follows from Theorem~\ref{thm:Ker-Q-in-terms-of-AR-n-plus-2-angles} that 
\[
\Ker Q 
	= \Braket{  [A'\tensor[]{]}{_{\CT}} -[A\tensor[]{]}{_{\CT}}  | 
	\begin{aligned}
		&\text{there are exchange triangles }
		\begin{tikzcd}[column sep=0.5cm,ampersand replacement=\&]
			Y\arrow{r} \& A' \arrow{r} \& \tensor*[]{Y}{^{*}} \arrow{r}\& \sus Y
		\end{tikzcd} \\[-7pt]
		&\text{and }
		\begin{tikzcd}[column sep=0.5cm,ampersand replacement=\&]
			\tensor*[]{Y}{^{*}}\arrow{r} \& A \arrow{r} \& Y \arrow{r}\& \sus \tensor*[]{Y}{^{*}}
		\end{tikzcd} 
		\text{for } Y\in\ind\CY
	\end{aligned}
	}.
\]
In particular, this is isomorphic to the subgroup
\[
N
	= \Braket{  [A'\tensor*[]{]}{_{\CT}^{\sp}} -[A\tensor*[]{]}{_{\CT}^{\sp}}  | 
	\begin{aligned}
		&\text{there are exchange triangles }
		\begin{tikzcd}[column sep=0.5cm,ampersand replacement=\&]
			Y\arrow{r} \& A' \arrow{r} \& \tensor*[]{Y}{^{*}} \arrow{r}\& \sus Y
		\end{tikzcd} \\[-7pt]
		&\text{and }
		\begin{tikzcd}[column sep=0.5cm,ampersand replacement=\&]
			\tensor*[]{Y}{^{*}}\arrow{r} \& A \arrow{r} \& Y \arrow{r}\& \sus \tensor*[]{Y}{^{*}}
		\end{tikzcd} 
		\text{for } Y\in\ind\CY
	\end{aligned}
	}
\]
of \(\tensor*[]{K}{_{0}^{\sp}}(\CT)\), 
using the isomorphism 
\(\tensor*[]{K}{_{0}^{\sp}}(\CT) \to \tensor[]{K}{_{0}}(\CC,\tensor[]{\BE}{_{\CT}},\tensor[]{\fs}{_{\CT}})\) sending \([T\tensor*[]{]}{_{\CT}^{\sp}}\) to \([T\tensor[]{]}{_{\CT}}\) (see \eqref{eqn:isomorphism-of-grothendieck-groups-for-cluster-tilting-T} from Section~\ref{sec:introduction} or Theorem~\ref{thm:analogue-of-PPPP-Prop-4-11}, or \cite[Prop.\ 4.11]{PadrolPaluPilaudPlamondon-associahedra-for-finite-type-cluster-algebras-and-minimal-relations-between-g-vectors}).
Therefore, we have isomorphisms
\begin{equation}
\label{eqn:chain-of-isomorphisms-for-CC-map}
\begin{tikzcd}
\tensor*[]{K}{_{0}^{\sp}}(\CT)/N \arrow{r}& \tensor[]{K}{_{0}}(\CC,\tensor[]{\BE}{_{\CT}},\tensor[]{\fs}{_{\CT}}) / \Ker Q \arrow{r}{\ol{\tensor*[]{Q}{_{\CX}^{\CT}}}}& \tensor[]{K}{_{0}}(\CC,\tensor[]{\BE}{_{\CX}},\tensor[]{\fs}{_{\CX}}). 
\end{tikzcd}
\end{equation}
The subgroup \(N\) was defined in \cite[Def.\ 2.4]{HolmJorgensen-generalized-friezes-and-a-modified-caldero-chapoton-map-depending-on-a-rigid-object-2}. 
Although in general the subgroup \(N\) depends on the choice of complement \(\CY\) to \(\CX\), the isomorphisms of \eqref{eqn:chain-of-isomorphisms-for-CC-map} show that the quotient \(\tensor*[]{K}{_{0}^{\sp}}(T) / N\) is independent of this choice. 
This was previously only known in the cases when \(\CC\) is the cluster category associated to a Dynkin-type \(\BA\) quiver \cite[Prop.~3.2.3]{Pescod-phd-thesis} or a Dynkin-type \(\BD\) quiver \cite[Thm.~5.1.7]{Pescod-phd-thesis}.
The authors thank Yann Palu for asking about the dependency of \(N\) on \(\CY\).

In \cite{HolmJorgensen-generalized-friezes-and-a-modified-caldero-chapoton-map-depending-on-a-rigid-object-1} and \cite{HolmJorgensen-generalized-friezes-and-a-modified-caldero-chapoton-map-depending-on-a-rigid-object-2}, 
Holm--J{\o{}}rgensen generalised the map of Caldero--Chapoton \cite{CalderoChapoton-cluster-algebras-as-hall-algebras-of-quiver-representations} in two ways, by defining their \emph{modified Caldero-Chapoton map}. 
This generalisation depends only on a rigid object and can take values in a general commutative ring \(A\). 
The modified Caldero-Chapoton map depends on several functions. 
In particular, it utilises a map \(\alpha \colon \obj(\CC) \to A\) given by \(\alpha(X) = \eps q ([\tensor*[]{T}{_{0}^{X}}\tensor*[]{]}{_{\CT}^{\sp}}-[\tensor*[]{T}{_{1}^{X}}\tensor*[]{]}{_{\CT}^{\sp}})\), where 
\(\eps \colon \tensor*[]{K}{_{0}^{\sp}}(\CT)/N \to A\) is a map 
satisfying some properties, \(q\colon \tensor*[]{K}{_{0}^{\sp}}(\CT) \to \tensor*[]{K}{_{0}^{\sp}}(\CT)/N\) is the canonical quotient, and \([\tensor*[]{T}{_{0}^{X}}\tensor*[]{]}{_{\CT}^{\sp}}-[\tensor*[]{T}{_{1}^{X}}\tensor*[]{]}{_{\CT}^{\sp}}\) is the classical Palu index of \(X\) with respect to \(\CT\).

When \(\CX = \CT\) is cluster tilting, we see that \(N = 0\) and \(q\) is the identity. Thus, in this special case, we input into \(\eps\) the genuine classical index of \(X\) with respect to \(\CT\). 
In the general case, although \(\tensor*[]{K}{_{0}^{\sp}}(\CT)/N\) works as the domain for \(\eps\), it does not explain \emph{why} it is needed. 
This is now explained by the isomorphism 
\(\tensor*[]{K}{_{0}^{\sp}}(\CT)/N \to \tensor[]{K}{_{0}}(\CC,\tensor[]{\BE}{_{\CX}},\tensor[]{\fs}{_{\CX}})\) given in \eqref{eqn:chain-of-isomorphisms-for-CC-map} and that the Grothendieck group \(\tensor[]{K}{_{0}}(\CC,\tensor[]{\BE}{_{\CX}},\tensor[]{\fs}{_{\CX}})\) is capturing the indices of objects in \(\CC\) with respect to \(\CX\). These observations led to the \emph{\(\CX\)-Caldero-Chapoton map} defined in \cite[Sec.\ 5]{JorgensenShah-grothendieck-groups-of-d-exangulated-categories-and-a-modified-CC-map}.

\end{rem}

{\setstretch{1}\begin{acknowledgements}
This work was supported by a DNRF Chair from the Danish National Research Foundation (grant DNRF156), by a Research Project 2 from the Independent Research Fund Denmark (grant 1026-00050B), by the Aarhus University Research Foundation (grant AUFF-F-2020-7-16), and by the Engineering and Physical Sciences Research Council (grant EP/P016014/1). 
The second author is also grateful to the London Mathematical Society for funding through an Early Career Fellowship with support from Heilbronn Institute for Mathematical Research (grant ECF-1920-57). 
\end{acknowledgements}}


\bibliographystyle{mybst}
\bibliography{references}
\end{document}